\documentclass[11pt]{amsart}
\usepackage{amssymb}

\makeatletter
\newcommand{\sumprime}{\if@display\sideset{}{'}\sum%
            \else\sum'\fi}
\makeatother

\begin{document}

\numberwithin{equation}{section}

% define theorem environments
\newtheorem{theorem}{Theorem}[section]
\newtheorem{proposition}[theorem]{Proposition}
\newtheorem{conjecture}[theorem]{Conjecture}
\def\theconjecture{\unskip}
\newtheorem{corollary}[theorem]{Corollary}
\newtheorem{lemma}[theorem]{Lemma}
\newtheorem{observation}[theorem]{Observation}
\newtheorem{definition}{Definition}
\numberwithin{definition}{section} %\def\thedefinition{\unskip}
\newtheorem{remark}{Remark}
\def\theremark{\unskip}
\newtheorem{question}{Question}
\def\thequestion{\unskip}
\newtheorem{example}{Example}
\def\theexample{\unskip}
\newtheorem{problem}{Problem}

\def\vvv{\ensuremath{\mid\!\mid\!\mid}}
\def\intprod{\mathbin{\lr54}}
\def\reals{{\mathbb R}}
\def\integers{{\mathbb Z}}
\def\N{{\mathbb N}}
\def\complex{{\mathbb C}\/}
\def\dist{\operatorname{dist}\,}
\def\spec{\operatorname{spec}\,}
\def\interior{\operatorname{int}\,}
\def\trace{\operatorname{tr}\,}
\def\cl{\operatorname{cl}\,}
\def\essspec{\operatorname{esspec}\,}
\def\range{\operatorname{\mathcal R}\,}
\def\kernel{\operatorname{\mathcal N}\,}
\def\dom{\operatorname{Dom}\,}
\def\linearspan{\operatorname{span}\,}
\def\lip{\operatorname{Lip}\,}
\def\sgn{\operatorname{sgn}\,}
\def\Z{ {\mathbb Z} }
\def\e{\varepsilon}
\def\p{\partial}
\def\rp{{ ^{-1} }}
\def\Re{\operatorname{Re\,} }
\def\Im{\operatorname{Im\,} }
\def\dbarb{\bar\partial_b}
\def\eps{\varepsilon}
\def\O{\Omega}
\def\Lip{\operatorname{Lip\,}}

\def\Hs{{\mathcal H}}
\def\E{{\mathcal E}}
\def\scriptu{{\mathcal U}}
\def\scriptr{{\mathcal R}}
\def\scripta{{\mathcal A}}
\def\scriptc{{\mathcal C}}
\def\scriptd{{\mathcal D}}
\def\scripti{{\mathcal I}}
\def\scriptk{{\mathcal K}}
\def\scripth{{\mathcal H}}
\def\scriptm{{\mathcal M}}
\def\scriptn{{\mathcal N}}
\def\scripte{{\mathcal E}}
\def\scriptt{{\mathcal T}}
\def\scriptr{{\mathcal R}}
\def\scripts{{\mathcal S}}
\def\scriptb{{\mathcal B}}
\def\scriptf{{\mathcal F}}
\def\scriptg{{\mathcal G}}
\def\scriptl{{\mathcal L}}
\def\scripto{{\mathfrak o}}
\def\scriptv{{\mathcal V}}
\def\frakg{{\mathfrak g}}
\def\frakG{{\mathfrak G}}

\def\ov{\overline}

\thanks{Research supported by the Key Program of NSFC No. 11031008.}

\address{Department of Applied Mathematics, Tongji University, Shanghai, 200092, China}
 \email{boychen@tongji.edu.cn}

\title{Weighted Bergman spaces and the $\bar{\partial}-$equation}
\author{Bo-Yong Chen}
\date{}
\maketitle

\bigskip

\centerline{\it Dedicated to Professor Jinhao Zhang on the occasion of his seventieth birthday}

\bigskip

\begin{abstract} We give a H\"ormander type $L^2-$estimate for the $\bar{\partial}-$equation with respect to the measure $\delta_\Omega^{-\alpha}dV$, $\alpha<1$, on any bounded pseudoconvex domain with $C^2-$boundary. Several applications to the function theory of weighed Bergman spaces $A^2_\alpha(\Omega)$ are given, including a corona type theorem, a Gleason type theorem,  together with a density theorem.  We investigate in particular the boundary behavior of functions in $A^2_\alpha(\Omega)$ by proving an analogue of the Levi problem for  $A^2_\alpha(\Omega)$ and giving an optimal Gehring type estimate for functions in $A^2_\alpha(\Omega)$. A vanishing theorem for $A^2_1(\Omega)$ is established for arbitrary bounded domains. Relations between the weighted Bergman kernel and the Szeg\"o kernel are also discussed.

\bigskip


\smallskip

\end{abstract}

\section{Introduction}

Let $\Omega\subset {\mathbb C}^n$ be a pseudoconvex domain and let $\varphi$ be a $C^2$ plurisubharmonic (psh) function on $\Omega$. A fundamental theorem of H\"ormander (cf. \cite{Hormander65,Hormander90}, see also \cite{AndreottiVesentini65, Demailly82})  states that for any $\bar{\partial}-$closed $(0,1)-$form $v$, there exists a solution $u$ to the equation $\bar{\partial}u=v$ such that
\begin{equation}\label{eq:cor1}
\int_\Omega |u|^2 e^{-\varphi} dV\le \int_\Omega |v|^2_{i\partial\bar{\partial}\varphi} e^{-\varphi} dV
\end{equation}
provided the right-hand side is finite.

In 1983, Donnelly-Fefferman \cite{DonnellyFefferman83} made a striking discovery that under certain condition, the $\bar{\partial}-$equation may have solutions of finite $L^2-$norm with some {\it non-psh} weight. Such a discovery was extended and simplified substantially by a number of mathematicians (see e.g. \cite{DiederichOhsawa95,Berndtsson96,BerndtssonCharpentier00,McNeal96,BoasStraube99}), now may be formulated as follows: if $\psi$ is another $C^2$ psh function on $\Omega$ satisfying $i\alpha\partial\bar{\partial}\psi\ge i \partial\psi\wedge\bar{\partial}\psi$ for some $0<\alpha<1$, then the $L^2(\Omega,\varphi)-$minimal solution of the $\bar{\partial}-$equation enjoys the estimate
\begin{equation}\label{eq:cor2}
\int_\Omega |u|^2 e^{\psi-\varphi} dV\le {\rm const}_\alpha \int_\Omega |v|^2_{i\partial\bar{\partial}(\varphi+\psi)} e^{\psi-\varphi} dV
\end{equation}
provided the right-hand side is finite. In particular, if we take $\psi=-\frac{\alpha}{\alpha_0}\log(-\rho)$, where $\rho$ is a negative $C^2$ psh function verifying $-\rho\asymp \delta_\Omega^{\alpha_0}$, $\alpha_0>\alpha>0$ and $\delta_\Omega$ is the boundary distance function, then \eqref{eq:cor2} implies
 \begin{equation}\label{eq:cor3}
 \int_\Omega |u|^2e^{-\varphi}\delta_\Omega^{-\alpha}dV\le {\rm const}_{\alpha,\Omega}\int_\Omega |v|^2_{i\partial\bar{\partial}\varphi}e^{-\varphi}\delta_\Omega^{-\alpha}dV,
 \end{equation}
 which has significant applications in the study of regularities of the Bergman projection (cf. \cite{BerndtssonCharpentier00}, see also \cite{MichelShaw01}).
 In case $\Omega$ has a $C^2-$boundary, Diederich-Forn{\ae}ss \cite{DiederichFornaess77} proved the existence of such a $\rho$, where $\alpha_0$ is called a Diederich-Forn{\ae}ss exponent. On the other side, there are pseudoconvex domains  (so-called worm domains) whose Diederich-Forn{\ae}ss exponents are arbitrarily small (cf. \cite{DiederichFornaess77MathAnn}).

 In this paper, we shall proving the following

\begin{theorem}\label{th:main}
  Let $\Omega\subset\subset {\mathbb C}^n$ be a pseudoconvex domain with $C^2-$boundary and $\varphi$ a $C^2$ psh function on $\Omega$. Then for each $\alpha<1$ and each $\bar{\partial}-$closed $(0,1)-$form $v$ with $
\int_\Omega |v|^2_{i\partial\bar{\partial}\varphi}e^{-\varphi}\delta_\Omega^{-\alpha}dV<\infty,
$
 there is a solution $u$ to the equation  $\bar{\partial}u=v$ such that \eqref{eq:cor3} holds.
 \end{theorem}

We shall give various applications of this result to the function theory of the weighted Bergman space $A^2_\alpha(\Omega)$, that is, the Hilbert space of holomorphic functions $f$ on $\Omega$ with
 $$
 \|f\|^2_\alpha:=\int_\Omega |f|^2 \delta_\Omega^{-\alpha} dV<\infty.
 $$
 The spaces $A_\alpha^2(\Omega)$ coincide with the usual Sobolev spaces of holomorphic functions for $\alpha<1$, i.e.,
 $$
 A_\alpha^2(\Omega)={\mathcal O}(\Omega)\cap W^\alpha(\Omega)
 $$
(see Ligocka \cite{LigockaHarmonic}). Despite of deep results achieved for strongly pseudoconvex domains (see e.g., \cite{Beatrous85,Englis08}), few progress has been made in the case of weakly pseudoconvex domains.

\begin{theorem}\label{th:corona}
Let $\Omega\subset\subset {\mathbb C}^n$ be a pseudoconvex domain with $C^2-$boundary. Let $f_1,f_2\in {\mathcal O}(\Omega)$ and $\delta>0$ be such that
$$
\delta^2\le |f_1|^2+|f_2|^2\le 1.
$$
Then for each $h\in A^2_\alpha(\Omega)$, $\alpha<1$, there are functions $g_1,g_2\in A^2_\alpha(\Omega)$ satisfying
$$
f_1g_1+f_2g_2=h.
$$
\end{theorem}

\begin{theorem}\label{th:gleason}
Let $\Omega\subset\subset {\mathbb C}^2$ be a pseudoconvex domain with $C^2-$boundary. If $w\in \Omega$ and $h\in A^2_\alpha(\Omega)$, $\alpha<1$, then there are functions $g_1,g_2\in A^2_\alpha(\Omega)$ satisfying $$
h(z)-h(w)=(z_1-w_1)g_1(z)+(z_2-w_2)g_2(z),\ \ \ \forall\,z\in \Omega.
$$
\end{theorem}

\begin{theorem}\label{th:density}
 Let $\Omega\subset\subset {\mathbb C}^n$ be a pseudoconvex domain with $C^2-$boundary.
\begin{enumerate}

 \item For each $\alpha<1$, $A^2_\alpha(\Omega)$ is dense in the space ${\mathcal O}(\Omega)$, equipped with the topology of uniform convergence on compact subsets.

 \item For any $\alpha_1<\alpha_2<1$, $A^2_{\alpha_2}(\Omega)$ is dense in $A^2_{\alpha_1}(\Omega)$.
  \end{enumerate}
\end{theorem}

The following result is an analogue of the Levi problem for $A^2_\alpha(\Omega)$, which also generalizes an old result of Pflug (cf. \cite{Pflug75}):

\begin{theorem}\label{th:levi}
 Let $\Omega\subset\subset {\mathbb C}^n$ be a pseudoconvex domain with $C^2-$boundary. Then for each $\alpha<1$, there are $\beta>0$ and $f\in A^2_\alpha(\Omega)$ such that for all $\zeta\in \partial \Omega$,
 $$
 \limsup_{z\rightarrow \zeta}\,|f(z)|\delta_\Omega(z)^{1-\frac{\alpha}2}\left|\log \delta_\Omega(z)\right|^\beta=\infty.
 $$
 \end{theorem}

 It should be pointed out that each bounded pseudoconvex domain with {\it $C^\infty-$boundary} is the domain of existence of a function in $A^\infty(\overline{\Omega}):={\mathcal O}(\Omega)\cap C^\infty(\overline{\Omega})$ (cf. \cite{Catlin80}, see also \cite{HakimSibony80}).

On the other side, we have the following Gehring type estimate:

\begin{theorem}\label{th:gehring}
 Let $\Omega\subset {\mathbb C}^n$ be a bounded domain with $C^2-$boundary and let $f\in A^2_\alpha(\Omega)$, $\alpha<1$. Then for almost all $\zeta\in \partial \Omega$
$$
|f(z)|={\rm o}(\delta_\zeta(z)^{-\frac{1-\alpha}2})\ \ \ \ \ \  {\it uniformly},
$$
as $z$ approaches $\zeta$ admissibly. Here $\delta_\zeta(z)=$minimum of $\delta_\Omega(z)$ and the distance from $z$ to the tangent space at $\zeta$, and $A={\rm o}(B)$ means $\lim A/B= 0$.
\end{theorem}

The concept of admissible approach was introduced by Stein \cite{Stein72} in his far-reaching generalization of Fatou's theorem for holomorphic functions in a bounded domain with $C^2-$boundary.

It turns out that the above bound is optimal for the case of the unit ball:

\begin{theorem}\label{th:ball}
Let ${\mathbb B}^n$ be the unit ball in ${\mathbb C}^n$ and ${\mathbb S}^n$ the unit sphere. For each $\alpha<1$,  there is a number $t_\alpha>1$ such that for each $\varepsilon>0$, there exists a  function $f\in A^2_\alpha({\mathbb B}^n)$ so that for each $\zeta\in {\mathbb S}^n$,
$$
\limsup |f(z)|(1-|z|)^{\frac{1-\alpha}2}\left|\log (1-|z|)\right|^{\frac{1+\varepsilon}2}>0
$$
as $z\rightarrow \zeta$ from the inside of the Koranyi region ${\mathcal A}_{t_\alpha}(\zeta)$ defined by
$$
{\mathcal A}_{t_\alpha}(\zeta)=\left\{z\in {\mathbb B}^n:|1-z\cdot\bar{\zeta}|<t_\alpha(1-|z|)\right\}.
$$
\end{theorem}

 Stein \cite{Stein72} suggested to study the relation between the Bergman and Szeg\"o kernels. In \cite{ChenFu11}, Chen-Fu obtained a comparison of the Szeg\"o and Bergman kernels for so-called $\delta-$regular domains including domains of finite type and domains with psh defining functions. Here we shall prove the following natural connection between the weighted Bergman kernelss $K_\alpha$ and the Szeg\"o kernel $S$, which seems not to have been noticed in the literature:

 \begin{theorem}\label{th:bergmanszego}
  Let $\Omega\subset {\mathbb C}^n$ be a bounded domain with $C^2-$boundary. Then
 $$
 (1-\alpha)^{-1}K_\alpha(z,w)\rightarrow S(z,w)
 $$
  locally uniformly in $z,w$ as $\alpha\rightarrow 1-$. In particular,
 $$
\left.\frac{\partial K_\alpha(z,w)}{\partial \alpha}\right|_{1-}:=\lim_{\alpha\rightarrow 1-}\frac{K_\alpha(z,w)-K_1(z,w)}{\alpha-1}=-S(z,w).
$$
\end{theorem}

For general bounded domains, a fundamental question immediately arises:

\medskip

{\it When is $A^2_\alpha(\Omega)$ trivial or nontrivial?  }

\medskip

Clearly, $A^2_\alpha(\Omega)$ is always nontrivial for $\alpha\le 0$. On the other side, we have the following vanishing theorem:

\begin{theorem}\label{th:vanish}
Let $\Omega$ be a bounded domain in ${\mathbb C}^n$.
\begin{enumerate}
\item For each $f\in {\mathcal O}(\Omega)$ with
$
\int_\Omega |f|^2 \delta_\Omega^{-1}\left(1+|\log \delta_\Omega|\right)^{-1}dV<\infty,
$
 we have $f=0$. In particular, $A^2_\alpha(\Omega)=\{0\}$ for each $\alpha\ge 1$.
\item
 Let $\Omega_\varepsilon=\{z\in \Omega:\delta_\Omega(z)>\varepsilon\}$ and let $c(\varepsilon):={\rm cap}\left(\overline{\Omega}_\varepsilon,\Omega\right)$ denote the capacity of $\overline{\Omega}_\varepsilon$ in $\Omega$. Suppose there is a sequence $\varepsilon_j\rightarrow 0+$, so that $c(\varepsilon_j)=O(\varepsilon_j^{-\alpha})$, then $A^2_\alpha(\Omega)=\{0\}$.
 \end{enumerate}
 \end{theorem}

As a consequence Theorem~\ref{th:vanish}, we have

\begin{theorem}\label{th:hyperconvex}
 Let $\Omega\subset {\mathbb C}^n$ be a bounded domain. For each $\varepsilon>0$, there does not exist a continuous psh function $\rho<0$ on $\Omega$ such that
 $$
 -\rho\le {\rm const}_\varepsilon \delta_\Omega\left(1+|\log \delta_\Omega|\right)^{-\varepsilon}.
 $$
 \end{theorem}

In particular, the order of hyperconvexity of $\Omega$ is no larger than $1$. In case $\partial\Omega$ is of class $C^2$, this result is a direct consequence of the Hopf lemma.

\section{Proof of Theorem~\ref{th:main}}

Let $\Omega\subset\subset {\mathbb C}^n$ be a pseudoconvex domain with $C^2-$boundary. Let $\varphi$ be a real-valued $C^2-$smooth function on $\Omega$. Let  $L^{p,q}_{(2)}(\Omega,\varphi)$ denote the space of $(p,q)-$forms $u$ on $\Omega$ satisfying
$$
\|u\|^2_{\varphi}:=\int_\Omega |u|^2 e^{-\varphi} dV<\infty.
$$
Let $\bar{\partial}^\ast_{\varphi}$ denote the adjoint of the operator $\bar{\partial}$ with respect to the corresponding inner product $(\cdot,\cdot)_{\varphi}$. We recall the the following twisted Morrey-Kohn-H\"ormander formula, which goes back to Ohsawa-Takegoshi (cf. \cite{OhsawaTakegoshi87,Berndtsson96,Siu96,McNeal96,Ohsawa01,BoasStraube99}):

\begin{proposition}\label{prop:twist}
 Let $\rho$ be a $C^2-$definining function of $\Omega$. Let $u$ be a $(0,1)-$form that is continuously differentiable on $\overline{\Omega}$ and satisfies the $\bar{\partial}-$Neumann boundary conditions on $\partial \Omega$, $\partial \rho \cdot u=0$, and let $\eta$ and $\varphi$ be real-valued functions that are twice continuously differentiable on $\overline{\Omega}$ with $\eta\ge 0$. Then
\begin{eqnarray*}
\|\sqrt{\eta}\bar{\partial}u\|_\varphi^2+\|\sqrt{\eta}\bar{\partial}^\ast_\varphi u\|_\varphi^2 & = & \sum_{j,k=1}^n \int_{\partial \Omega} \eta \frac{\partial^2\rho}{\partial z_j\partial\bar{z}_k}u_j\bar{u}_k e^{-\varphi}\frac{d\sigma}{|\nabla \rho|}+\sum_{j=1}^n \int_\Omega \eta\left|\frac{\partial u_j}{\partial \bar{z}_j}\right|^2e^{-\varphi} dV\\
&& +\sum_{j,k=1}^n \int_\Omega \left(\eta\frac{\partial^2\varphi}{\partial z_j\partial\bar{z}_k}-\frac{\partial^2\eta}{\partial z_j\partial\bar{z}_k}\right)u_j\bar{u}_k e^{-\varphi} dV\\
&& +2{\rm Re} \int _\Omega (\partial \eta \cdot u) \overline{\bar{\partial}^\ast_\varphi u} e^{-\varphi} dV.
\end{eqnarray*}
\end{proposition}

Now we prove Theorem~\ref{th:main}. It is well-known that {\it locally} the Diederich-Forn{\ae}ss exponents can be arbitrarily close to $1$ (cf. \cite{DiederichFornaess77}, Remark b), p. 133). Thus for any given $\alpha<1$, there exists a cover $\{U_j\}_{1\le j\le m_\alpha}$ of $\partial \Omega$ and $C^2$ psh functions $\rho_j<0$ on $\Omega\cap U_j$ such that
$$
C^{-1}\delta_\Omega(z)^{\frac{\alpha+1}2}\le -\rho_j(z)\le C \delta_\Omega(z)^{\frac{\alpha+1}2},\ \ \ z\in \Omega\cap U_j,\ 1\le j\le m_\alpha
$$
(Throughout this section, $C$ denotes a generic positive constant depending only on $\alpha$ and $\Omega$).
Take an open subset $U_0\subset\subset \Omega$ such that $\{U_j\}_{0\le j\le m_\alpha}$ forms a cover of $\overline{\Omega}$. Clearly, we can take a negative $C^2$
psh function $\rho_0$ on $U_0$ such that
$$
C^{-1}\delta_\Omega(z)^{\frac{\alpha+1}2}\le -\rho_0(z)\le C \delta_\Omega(z)^{\frac{\alpha+1}2},\ \ \ z\in  U_0
$$
(for example, $\rho_0(z)=|z|^2-\sup_\Omega |z|^2-1$).

Put $\varphi_\tau(z)=\varphi(z)+\tau|z|^2$, $\tau>0$, and $\Omega_\varepsilon:=\{z\in \Omega:\delta_\Omega(z)>\varepsilon\}$, $\varepsilon\ll 1$. By Proposition 2.1, we have
\begin{eqnarray}
&&\int_{\Omega_\varepsilon}(\eta+c(\eta)^{-1})|\bar{\partial}^\ast_{\varphi_{\tau}} w|^2 e^{-\varphi_{\tau}}dV+\int_{\Omega_\varepsilon}\eta |\bar{\partial}w|^2 e^{-\varphi_\tau}dV\nonumber\\
& \ge & \sum_{k,l}\int_{\Omega_\varepsilon} \left(\eta\frac{\partial^2\varphi_\tau}{\partial z_k{\partial}\bar{z}_l}-\frac{\partial^2\eta}{\partial z_k{\partial}\bar{z}_l}\right) w_k\bar{w}_l\ e^{-\varphi_\tau} dV-\int_{\Omega_\varepsilon}c(\eta)\left|\sum_k \frac{\partial\eta}{\partial z_k}w_k\right|^2 e^{-\varphi_\tau}dV
\end{eqnarray}
where $w=\sum_k w_k d\bar{z}_k$ lies in ${\rm Dom\,}\bar{\partial}^\ast_{\varphi_\tau}$  and is continuously differentiable on $\overline{\Omega}_\varepsilon$ (i.e., it satisfies the $\bar{\partial}-$Neumann boundary condition on $\partial\Omega_\varepsilon$), $\eta\ge 0$, $\eta\in C^2(\Omega)$ and $c$ is a positive continuous function on ${\mathbb R}^+$.

Let $\{\chi_j\}_{0\le j\le m_\alpha}$ be a partition of unity subordinate to the cover $\{U_j\}_{0\le j\le m_\alpha}$ of $\overline{\Omega}$. The point is that $w^j=\chi_j w$ still lies in ${\rm Dom\,}\bar{\partial}^\ast_{\varphi_\tau}$. Now we choose a real-valued function $\tilde{\chi}_j\in C_0^\infty(U_j)$ so that $\tilde{\chi}_j=1$ on ${\rm supp\,}\chi_j$.
Put $\psi_j=-\frac{2\alpha}{\alpha+1}\log(-\rho_j)$. Applying (2.1) to each $w^j$ with $\eta=e^{-\tilde{\chi}_j\psi_j}$ and $c(\eta)=\frac{1-\alpha}{2\alpha}e^{\tilde{\chi}_j\psi_j}$, we get
\begin{eqnarray*}
&& \sum_{k,l}\int_{\Omega_\varepsilon\cap U_j} \frac{\partial^2\varphi_\tau}{\partial z_k{\partial}\bar{z}_l}|\chi_j|^2 w_k\bar{w}_l\ e^{-\varphi_\tau-\psi_j} dV\\
& \le & \int_{\Omega_\varepsilon\cap U_j}|\bar{\partial} (\chi_j w)|^2 e^{-\varphi_\tau-\psi_j}dV+\frac{1+\alpha}{1-\alpha}\int_{\Omega_\varepsilon\cap U_j} |\bar{\partial}^\ast_{\varphi_\tau} (\chi_j w)|^2 e^{-\varphi_\tau-\psi_j}dV
\end{eqnarray*}
because
$$
-i(\partial\bar{\partial}\eta+c(\eta)\partial\eta\wedge \bar{\partial}\eta) = i e^{-\psi_j}\left(\partial\bar{\partial}\psi_j-\frac{\alpha+1}{2\alpha}\partial\psi_j\wedge \bar{\partial}\psi_j\right)\ge 0
$$
 holds on $\Omega\cap {\rm supp\,}\chi_j$.
Since $e^{-\psi_j} \asymp \delta_\Omega^\alpha$ on $\Omega\cap U_j$, we get
\begin{equation}
 \sum_{k,l}\int_{\Omega_\varepsilon\cap U_j} \frac{\partial^2\varphi_\tau}{\partial z_k{\partial}\bar{z}_l}|\chi_j|^2 w_k\bar{w}_l\ e^{-\varphi_\tau}\delta_\Omega^\alpha dV\le C\int_{\Omega_\varepsilon\cap U_j}(|\bar{\partial} (\chi_j w)|^2+|\bar{\partial}^\ast_{\varphi_\tau} (\chi_j w)|^2) e^{-\varphi_\tau}\delta_\Omega^\alpha dV.
\end{equation}
Thus
\begin{eqnarray*}
&& \sum_{k,l}\int_{\Omega_\varepsilon} \frac{\partial^2\varphi_\tau}{\partial z_k{\partial}\bar{z}_l} w_k\bar{w}_l\ e^{-\varphi_\tau}\delta_\Omega^\alpha dV\\
& =  & \sum_{k,l}\int_{\Omega_\varepsilon} \frac{\partial^2\varphi_\tau}{\partial z_k{\partial}\bar{z}_l}\left(\sum_{j=0}^{m_\alpha} \chi_j\right)^2 w_k \bar{w}_l e^{-\varphi_\tau}\delta_\Omega^\alpha dV \\
& \le & (m_\alpha+1)\sum_{j=0}^{m_\alpha} \sum_{k,l}\int_{\Omega_\varepsilon\cap U_j} \frac{\partial^2\varphi_\tau}{\partial z_k{\partial}\bar{z}_l}|\chi_j|^2 w_k\bar{w}_l\ e^{-\varphi_\tau}\delta_\Omega^\alpha dV\\
& \le & (m_\alpha+1)C\sum_{j=0}^{m_\alpha}\int_{\Omega_\varepsilon\cap U_j}(|\bar{\partial} (\chi_j w)|^2+|\bar{\partial}^\ast_{\varphi_\tau} (\chi_j w)|^2) e^{-\varphi_\tau}\delta_\Omega^\alpha dV
\end{eqnarray*}
by (2.2).
Since
$$
\bar{\partial}(\chi_j w)=\chi_j \bar{\partial}w+\bar{\partial}\chi_j\wedge w,\ \ \ \bar{\partial}^\ast_{\varphi_\tau}(\chi_j w)=\chi_j \bar{\partial}^\ast_{\varphi_\tau}w-\bar{\partial}\chi_j\lrcorner w,
$$
thus by Schwarz's inequality,
\begin{eqnarray}
&& \sum_{k,l}\int_{\Omega_\varepsilon} \frac{\partial^2\varphi_\tau}{\partial z_k{\partial}\bar{z}_l} w_k\bar{w}_l\ e^{-\varphi_\tau}\delta_\Omega^\alpha dV\nonumber\\
& \le & 2 (m_\alpha+1)C\sum_{j=0}^{m_\alpha}\int_{\Omega_\varepsilon\cap U_j}(|\bar{\partial}  w|^2+|\bar{\partial}^\ast_{\varphi_\tau}  w|^2+2|w|^2 |\bar{\partial}\chi_j|^2)e^{-\varphi_\tau}\delta_\Omega^{\alpha} dV\nonumber\\
& \le & 2 (m_\alpha+1)^2C\int_{\Omega_\varepsilon}(|\bar{\partial}  w|^2+|\bar{\partial}^\ast_{\varphi_\tau}  w|^2)e^{-\varphi_\tau}\delta_\Omega^{\alpha}\nonumber dV\\
&&+4(m_\alpha+1)C\int_{\Omega_\varepsilon} |w|^2 \sum_j |\bar{\partial}\chi_j|^2e^{-\varphi_\tau}\delta_\Omega^{\alpha} dV.
\end{eqnarray}
Since $\partial\bar{\partial}\varphi_\tau=\partial\bar{\partial}\varphi+\tau\partial\bar{\partial}|z|^2$, thus when $\tau=\tau(\alpha,\Omega)$ is sufficiently large, the term in (2.3) may be absorbed by the left-hand side and we get the following basic inequality
\begin{equation}
\sum_{k,l}\int_{\Omega_\varepsilon} \frac{\partial^2\varphi}{\partial z_k{\partial}\bar{z}_l} w_k\bar{w}_l\ e^{-\varphi_\tau}\delta_\Omega^\alpha dV\le C\int_{\Omega_\varepsilon}(|\bar{\partial}  w|^2+|\bar{\partial}^\ast_{\varphi_\tau}  w|^2)e^{-\varphi_\tau}\delta_\Omega^{\alpha} dV.
\end{equation}
The remaining argument is standard. By H\"ormander \cite{Hormander65}, Proposition 2.1.1, the same inequality holds for any $w\in L^{0,1}_{(2)} (\Omega_\varepsilon,\varphi_\tau)\cap  {\rm Dom\,}\bar{\partial}\,\cap {\rm Dom\,}\bar{\partial}^\ast_{\varphi_\tau}$ (Note that  $C_\varepsilon^{-1}\le \delta_\Omega^\alpha \le C_\varepsilon$ on $\Omega_\varepsilon$). In particular, if $\bar{\partial}w=0$, then
$$
\sum_{k,l}\int_{\Omega_\varepsilon} \frac{\partial^2\varphi}{\partial z_k{\partial}\bar{z}_l} w_k\bar{w}_l\ e^{-\varphi_\tau}\delta_\Omega^\alpha dV\le C\int_{\Omega_\varepsilon} |\bar{\partial}^\ast_{\varphi_\tau}  w|^2 e^{-\varphi_\tau}\delta_\Omega^{\alpha} dV.
$$
By Schwarz's inequality,
\begin{eqnarray*}
\left|\int_{\Omega_\varepsilon} \langle {v},w\rangle e^{-\varphi_\tau}dV\right|^2 & \le & \int_{\Omega_\varepsilon} |v|^2_{i\partial\bar{\partial}\varphi}e^{-\varphi_\tau}\delta_\Omega^{-\alpha} dV\sum_{k,l}\int_{\Omega_\varepsilon} \frac{\partial^2\varphi}{\partial z_k{\partial}\bar{z}_l} w_k\bar{w}_l\ e^{-\varphi_\tau}\delta_\Omega^\alpha dV\\
& \le & C \int_\Omega |v|^2_{i\partial\bar{\partial}\varphi}e^{-\varphi_\tau}\delta_\Omega^{-\alpha} dV\int_{\Omega_\varepsilon} |\bar{\partial}^\ast_{\varphi_\tau}  w|^2 e^{-\varphi_\tau}\delta_\Omega^{\alpha} dV.
\end{eqnarray*}
For general $w\in {\rm Dom\,}\bar{\partial}^\ast_{\varphi_\tau}$, one has the orthogonal decomposition $w=w_1+w_2$ where $w_1\in {\rm Ker\,}\bar{\partial}$ and $w_2\in ({\rm Ker\,}\bar{\partial})^\bot\subset {\rm Ker\,}\bar{\partial}^\ast_{\varphi_\tau}$. Thus
\begin{eqnarray*}
&&\left|\int_{\Omega_\varepsilon} \langle {v},w\rangle e^{-\varphi_\tau}dV\right|^2 =  \left|\int_{\Omega_\varepsilon} \langle {v},w_1\rangle e^{-\varphi_\tau}dV\right|^2\\
& \le & C \int_\Omega |v|^2_{i\partial\bar{\partial}\varphi}e^{-\varphi_\tau}\delta_\Omega^{-\alpha} dV\int_{\Omega_\varepsilon} |\bar{\partial}^\ast_{\varphi_\tau}  w_1|^2 e^{-\varphi_\tau}\delta_\Omega^{\alpha} dV \\
& = & C \int_\Omega |v|^2_{i\partial\bar{\partial}\varphi}e^{-\varphi_\tau}\delta_\Omega^{-\alpha} dV\int_{\Omega_\varepsilon} |\bar{\partial}^\ast_{\varphi_\tau}  w|^2 e^{-\varphi_\tau}\delta_\Omega^{\alpha} dV.
\end{eqnarray*}
Applying the Hahn-Banach theorem to the anti-linear map
$$
\delta_\Omega^{\frac{\alpha}2}\bar{\partial}^\ast_{\varphi_\tau}w\mapsto \int_{\Omega_\varepsilon} \langle {v},w\rangle e^{-\varphi_\tau}dV
$$
together with the Riesz representation theorem, we get a solution $u_\varepsilon$ of the equation
$
\bar{\partial}(\delta_\Omega^{\frac{\alpha}2}{u}_\varepsilon)={v}
$
on $\Omega_\varepsilon$
with the estimate
$$
\int_{\Omega_\varepsilon} |{u}_\varepsilon|^2 e^{-\varphi_\tau}dV\le C\int_\Omega |v|^2_{i\partial\bar{\partial}\varphi}e^{-\varphi_\tau}\delta_\Omega^{-\alpha} dV.
$$
Taking a weak limit of  $\delta_\Omega^{\frac{\alpha}2} u_\varepsilon$ as $\varepsilon\rightarrow 0+$, we immediately obtain the desired solution. Q.E.D.

\begin{remark}\label{rk:kohn}

\begin{enumerate}
\item The additional weight $t|z|^2$ is somewhat inspired by Kohn \cite{Kohn73}.

 \item The following variation of Theorem 1.1 is more convenient for applications, which may be proved similarly, together with an additional approximation argument.
\end{enumerate}
\end{remark}

\begin{theorem}\label{th:main2}
Let $\Omega\subset\subset {\mathbb C}^n$ be a pseudoconvex domain with $C^2-$boundary and let $\hat{\Omega}\subset \Omega$ be a pseudoconvex domain. Let $\varphi$  be a  psh function on $\hat{\Omega}$ such that $i\partial\bar{\partial}\varphi\ge i\partial\bar{\partial}\psi$ in the sense of distribution, where $\psi$ is a $C^2$ psh function on $\hat{\Omega}$. Then for each $\alpha<1$ and each $\bar{\partial}-$closed $(0,1)-$form $v$ with $
\int_{\hat{\Omega}} |v|^2_{i\partial\bar{\partial}\psi}e^{-\varphi}\delta_\Omega^{-\alpha}dV<\infty,
$
there is a solution $u$ to the equation  $\bar{\partial}u=v$ on $\hat{\Omega}$ such that
$$
\int_{\hat{\Omega}} |u|^2e^{-\varphi}\delta_\Omega^{-\alpha}dV\le {\rm const}_{\alpha,\Omega}\int_{\hat{\Omega}} |v|^2_{i\partial\bar{\partial}\psi}e^{-\varphi}\delta_\Omega^{-\alpha}dV.
$$
\end{theorem}

\section{Some consequences of Theorem 1.1}

3.1. We first prove Theorem~\ref{th:corona}. Following Wolff's approach to Carleson's theorem (cf. \cite{Garnett07}, p. 315), we put
$$
g_1=h\frac{\bar{f}_1}{|f|^2}-uf_2,\ \ \ g_2=h\frac{\bar{f}_2}{|f|^2}+uf_1
$$
where $|f|^2=|f_1|^2+|f_2|^2$. Clearly, $f_1g_1+f_2g_2=h$, so the problem is reduced to choose $u\in L^2_\alpha(\Omega)$, i.e., $\int_\Omega |u|^2 \delta_\Omega^{-\alpha}dV<\infty$, so that $g_1,g_2$ are holomorphic. Thus it suffices to solve
$$
\bar{\partial}u=h\frac{\overline{f_2\partial f_1}-\overline{f_1\partial f_2}}{|f|^4}=:v
$$
such that $u\in L^2_\alpha(\Omega)$. Applying Theorem~\ref{th:main} with $\varphi=\log |f|^2$, we get a solution $u$ satisfying
$$
\int_\Omega |u|^2|f|^{-2}\delta_\Omega^{-\alpha}dV\le {\rm const}_{\alpha,\Omega}\int_\Omega |v|^2_{i\partial\bar{\partial}\varphi}|f|^{-2}\delta_\Omega^{-\alpha}dV.
$$
A straightforward calculation shows
$$
\partial\bar{\partial}\varphi=\frac{(f_1\partial f_2-f_2\partial f_1)\wedge \overline{(f_1\partial f_2-f_2\partial f_1)}}{|f|^4}
$$
so that $|v|^2_{i\partial\bar{\partial}\varphi}\le  |h|^2/|f|^{4}\le |h|^2/\delta^{4}$. Thus
$$
\int_\Omega |u|^2 \delta_\Omega^{-\alpha}dV\le {\rm const}_{\alpha,\Omega}\, \delta^{-6}\int_\Omega |h|^2\delta_\Omega^{-\alpha} dV.
$$
Q.E.D.

\medskip

 3.2. Next we prove Theorem~\ref{th:gleason}. The argument is a slightly modification of 3.1. Without loss of generality, we assume $w=0$, $h(0)=0$, $|z|^2<e^{-1}$ on $\Omega$. Put $f_k=z_k$, $k=1,2$ and $\varphi=-\log(-\log|f|^2)$. Then we have
$$
\partial\bar{\partial}\varphi\ge \frac{(f_1\partial f_2-f_2\partial f_1)\wedge \overline{(f_1\partial f_2-f_2\partial f_1)}}{|f|^4(-\log|f|^2)}.
$$
Let $g_k$, $v$ be defined as above and put
$\hat{\Omega}=\Omega\backslash \{f_1=0\}$. By Theorem~\ref{th:main2}, we may solve the equation $\bar{\partial}u=v$ on $\hat{\Omega}$ such that
$$
\int_{\hat{\Omega}} |u|^2 \delta_\Omega^{-\alpha} dV\le \int_{\hat{\Omega}} |u|^2e^{-\varphi}\delta^{-\alpha}_\Omega dV\le {\rm const}_{\alpha,\Omega}\int_{\hat{\Omega}} |v|^2_{i\partial\bar{\partial}\varphi} e^{-\varphi}\delta_\Omega^{-\alpha} dV
$$
since the last term is bounded by
\begin{eqnarray*}
&& {\rm const}_{\alpha,\Omega}\int_{\hat{\Omega}}|h|^2|f|^{-4}(\log |f|^2)^2 \delta_\Omega^{-\alpha} dV\\
& = & {\rm const}_{\alpha,\Omega}\int_{\hat{\Omega}\cap \{|z|<\varepsilon\}}|h|^2|f|^{-4}(\log |f|^2)^2 \delta_\Omega^{-\alpha} dV+{\rm const}_{\alpha,\Omega}\int_{\hat{\Omega}\backslash \{|z|<\varepsilon\}}|h|^2|f|^{-4}(\log |f|^2)^2 \delta_\Omega^{-\alpha} dV\\
& \le & {\rm const}_{\alpha,\Omega} \int_{\{|z|<\varepsilon\}} |z|^{-2}(\log |z|)^2dV+{\rm const}_{\alpha,\Omega}\int_{\Omega} |h|^2\delta_\Omega^{-\alpha} dV<\infty
\end{eqnarray*}
where $\varepsilon>0$ is so small that $\{|z|\le \varepsilon\}\subset \Omega$. Thus $g_1,g_2$ are holomorphic on $\hat{\Omega}$ such that
$$
\int_{\hat{\Omega}}|g_k|^2\delta_\Omega^{-\alpha}dV<\infty, \ \ \ k=1,2.
$$
The assertion follows immediately from Riemann's removable singularities theorem. Q.E.D.

\begin{remark}\label{rk:hormander}
It is possible to extend both the Corona and Gleason type theorems to general cases by using the Koszul complex technique introduced by H\"ormander \cite{Hormander67a}. But the argument will be substantially longer and not very enlightening, so that we shall not treat here.
\end{remark}

3.3. Finally, we prove Theorem~\ref{th:density}. (a) Let $K$ be a compact subset of $\Omega$ and $f\in {\mathcal O}(\Omega)$. We take a strictly psh exhaustion function $\psi\in C^\infty(\Omega)$ such that $K\subset \{\psi<0\}$. Let $\kappa$ be a $C^\infty$ convex increasing function such that $\kappa=0$ on $(-\infty,0]$ and $\kappa'>0$, $\kappa''>0$ on $(0,+\infty)$. Let $\rho<0$ be a bounded strictly psh exhaustion function on $\Omega$. Choose $\varepsilon>0$ so small that $\{\psi\le 0\}\subset \{\rho<-\varepsilon\}$. Let $\chi\in C^\infty_0(\Omega)$ be a real-valued function satisfying $\chi=1$ in a neighborhood of $\{\rho\le -\varepsilon\}$.   We construct a $2-$parameter family of weight functions as follows
$$
\varphi_{t,s}(z)=|z|^2+t\chi(z)\kappa( \psi(z))+s\kappa(\rho(z)+\varepsilon),\ \ \ t,s>0.
$$
It is easy to see that for any $t>0$ there is a sufficiently large number $s=s(t)>0$ such that $\partial\bar{\partial}\varphi_{t,s}\ge \partial\bar{\partial}|z|^2$. Let $\hat{\chi}\in C^\infty_0(\Omega)$ such that $\hat{\chi}=1$ in a neighborhood of $\{\psi\le 0\}$ and $\hat{\chi}(z)=0$ if $\rho(z)\ge -\varepsilon$. By Theorem 1.1, we may solve the equation
$$
\bar{\partial}u_t=f\bar{\partial}\hat{\chi}
$$
such that
\begin{eqnarray*}
\int_\Omega |u_t|^2 e^{-\varphi_{t,s}}\delta_\Omega^{-\alpha}dV &\le & {\rm const}_{\alpha,\Omega}\int_\Omega |f|^2|\bar{\partial}\hat{\chi}|^2 e^{-\varphi_{t,s}}\delta_\Omega^{-\alpha}dV\\
& \le &
 {\rm const}_{\alpha,\Omega}\int_{{\rm supp\,}\bar{\partial}\hat{\chi}}|f|^2e^{-t\kappa\circ \psi}\delta_\Omega^{-\alpha}dV\rightarrow 0
\end{eqnarray*}
as $t\rightarrow +\infty$. Since $\varphi_{t,s}(z)=|z|^2$ whenever $\psi(z)\le 0$, we conclude that
$$
\int_{\{\psi\le 0\}} |u_t|^2dV\rightarrow 0
$$
as $t\rightarrow +\infty$, so is the function $f_t-f$ where $f_t:=\hat{\chi}f-u_t$. On the other hand, $f_t\in A^2_\alpha(\Omega)$ because $\varphi_{t,s}$ is a bounded function. Since $f_t-f$ is holomorphic on $\{\psi<0\}$, a standard compactness argument yields
$$
\sup_K |f_t-f|\rightarrow 0
$$
as $t\rightarrow +\infty$.

(b) We take a $C^2$ psh function $\rho<0$ on $\Omega$ such that $-\rho\asymp \delta_\Omega^a$ for some $a>0$. Let $0\le \tilde{\chi}\le 1$ be a cut-off function on ${\mathbb R}$ such that $\tilde{\chi}|_{(-\infty,-\log 2)}=1$ and $\tilde{\chi}|_{(0,\infty)}=0$. Let $f\in A^2_{\alpha_1}(\Omega)$ be given. For each $\varepsilon>0$, we define
$$
v_\varepsilon=f\bar{\partial}\tilde{\chi}(-\log(-\rho+\varepsilon)+\log 2\varepsilon),\ \ \ \varphi_\varepsilon=-\frac{\alpha_2-\alpha_1}a \log(-\rho+\varepsilon).
$$
By Theorem 1.1, we have a solution of $\bar{\partial}u_\varepsilon=v_\varepsilon$ so that
\begin{eqnarray*}
\int_\Omega |u_\varepsilon|^2 e^{-\varphi_\varepsilon} \delta_\Omega^{-\alpha_2}dV & \le & {\rm const.}\int_\Omega |v_\varepsilon|^2_{i\partial\bar{\partial}\varphi_\varepsilon}e^{-\varphi_\varepsilon}\delta_\Omega^{-\alpha_2}dV\\
& \le & {\rm const.} \int_{\varepsilon\le -\rho\le 3\varepsilon}|f|^2 \delta_\Omega^{-\alpha_1}dV
\end{eqnarray*}
for $i\partial\bar{\partial}\varphi_\varepsilon\ge \frac{\alpha_2-\alpha_1}a i\partial \log (-\rho+\varepsilon)\wedge \bar{\partial}\log (-\rho+\varepsilon)$. Put
$$
f_\varepsilon=f\tilde{\chi} (-\log(-\rho+\varepsilon)+\log 2\varepsilon)-u_\varepsilon.
$$
Since $\varphi_\varepsilon$ is bounded and
$$
e^{-\varphi_\varepsilon}\ge e^{\frac{\alpha_2-\alpha_1}a\log (-\rho)}\asymp \delta_\Omega^{\alpha_2-\alpha_1},
$$
we conclude that $f_\varepsilon\in A^2_{\alpha_2}(\Omega)$ and
\begin{eqnarray*}
\int_\Omega |f_\varepsilon-f|^2 \delta_\Omega^{-\alpha_1}dV & \le & 2 \int_{-\rho\le 3\varepsilon}|f|^2 \delta_\Omega^{-\alpha_1}dV+2\int_\Omega |u_\varepsilon|^2\delta_\Omega^{-\alpha_1}dV\\
& \le & 2 \int_{-\rho\le 3\varepsilon}|f|^2 \delta_\Omega^{-\alpha_1}dV+{\rm const.}\int_\Omega |u_\varepsilon|^2 e^{-\varphi_\varepsilon}\delta_\Omega^{-\alpha_2}dV\\
& \le & 2 \int_{-\rho\le 3\varepsilon}|f|^2 \delta_\Omega^{-\alpha_1}dV+ {\rm const.} \int_{\varepsilon\le -\rho\le 3\varepsilon}|f|^2 \delta_\Omega^{-\alpha_1}dV\\
& \rightarrow & 0
\end{eqnarray*}
as $\varepsilon\rightarrow 0+$. Q.E.D.

\begin{problem}\label{prob:dense}
 Is the Hardy space $H^2(\Omega)$ dense in $A^2_\alpha(\Omega)$ for each $\alpha<1$?
\end{problem}

\begin{remark}\label{rk:regeree}
 The referee of this paper pointed out the following
\begin{enumerate}

\item Bell and Boas have proved a theorem related to Theorem~\ref{th:density} (cf. \cite{BellBoas}, Theorem 1\,).

 \item There is a standard argument as follows,  which is perhaps more straightforward than the author's proof. Choose a cover $\{U_j\}_{j=1}^m$ of the boundary and vectors $n_j$ such that $z-\varepsilon n_j\in \Omega$ for $1\le j\le m$, $z\in U_j$, $\varepsilon\le \varepsilon_0$. Choose $\phi_0\in C^\infty_0(\Omega)$ and $\phi_j\in C^\infty_0(U_j)$, $1\le j\le m$, with $\sum \phi_j=1$ in a neighborhood of $\overline{\Omega}$. Set
$$
f_\varepsilon(z)=\phi_0(z)f(z)+\sum_{j=1}^m \phi_j(z)f(z-\varepsilon n_j).
$$
Then $f_\varepsilon\rightarrow f$ in the norm with weight $\delta_\Omega^{-\alpha}$.
The theorem now follows by correcting $f_\varepsilon$ via
$$
\bar{\partial}f_\varepsilon=f\bar{\partial}\phi_0+\sum_{j=1}^m f(z-\varepsilon n_j)\bar{\partial}\phi_j=\sum_{j=1}^m [f(z-\varepsilon n_j)-f(z)]\bar{\partial}\phi_j
$$
(because $\sum_{j=0}^m \bar{\partial}\phi_j=0$ on $\Omega$). The norm of the right hand side tends to zero; so if we solve the $\bar{\partial}-$equation with the estimate that was shown, the corrections we make to the $f_\varepsilon$ tend to zero as well in norm, and we are done.
\end{enumerate}
\end{remark}

\section{Proof of Theorem~\ref{th:levi}}

4.1. Let $\Omega\subset {\mathbb C}^n$ be a bounded domain. We define the pluricomplex Green function $g_\Omega(\cdot,w)$ with pole at $w\in \Omega$ as
$$
g_\Omega(z,w)=\sup\left\{u(z):u\in PSH(\Omega),u<0,\limsup_{z\rightarrow w}\,(u(z)-\log|z-w|)<\infty\right\}.
$$
It is well-known that $g_\Omega(\cdot,w)\in PSH(\Omega)$ for each fixed $w$ and $g_\Omega\in C(\overline{\Omega}\times {\Omega}\backslash\{z=w\})$ when $\Omega$ is hyperconvex (cf. \cite{Klimek91}). We need the following  estimate of $g_\Omega$ due to Blocki \cite{Blocki04}:

\begin{theorem}\label{th:blocki}
 Let $\Omega\subset\subset {\mathbb C}^n$ be a pseudoconvex domain. Suppose there is a negative psh function $\rho$ on $\Omega$ satisfying
$$
C_1 \delta_\Omega^a(z)\le -\rho(z)\le C_2\delta_\Omega^b(z),\ \ \ z\in \Omega
$$
 where $C_1,C_2>0$ and $a\ge b\ge 0$ are constants. Then there are positive numbers $\delta_0,C$ such that $$
\{g_\Omega(\cdot,w)\le -1\}\subset \{C^{-1}\delta_\Omega(w)^{\frac{a}b}|\log \delta_\Omega(w)|^{-\frac1b}\le \delta_\Omega\le C\delta_\Omega(w)^{\frac{b}a}|\log\delta_\Omega(w)|^{\frac{n}a}\}
$$
 holds for any $w\in \Omega$ with $\delta_\Omega(w)\le \delta_0$.
 \end{theorem}

4.2. Let $K_\alpha$ be the Bergman kernel of $A^2_\alpha(\Omega)$.

\begin{proposition}\label{prop:pflug}
Suppose $\lim_{z\rightarrow \partial\Omega} K_{\alpha}(z)\eta(z)= \infty$ where $\eta$ is a positive continuous function on $\Omega$. Then there exists a function $f\in A^2_\alpha(\Omega)$ such that
$$
\limsup_{z\rightarrow \zeta}\,|f(z)|\sqrt{\eta(z)}=\infty,\ \ \  \forall\,\zeta\in \partial\Omega.
$$
\end{proposition}

\begin{proof} The argument is standard (see e.g. \cite{JarnickiPflug00}, p. 416--417).  We claim that the following assertion holds:

{\it For each $\zeta\in \partial \Omega$ and each sequence of points in $\Omega$ with $z_j\rightarrow \zeta$, there exists a function $f\in A^2_\alpha(\Omega)$ such that $\sup_j|f(z_j)|\sqrt{\eta(z_j)}=\infty$.}

Suppose there is a point $\zeta\in \partial \Omega$ and a sequence of points in $\Omega$ such that $z_j\rightarrow \zeta$ such that $\sup_j|f(z_j)|\sqrt{\eta(z_j)}<\infty$, $\forall\,f\in A^2_\alpha(\Omega)$. Applying the Banach-Steinhaus theorem to the linear functional $f\rightarrow f(z_j)\sqrt{\eta(z_j)}$, we get
$$
\sup_j|f(z_j)|\sqrt{\eta(z_j)}\le {\rm const.}\|f\|
$$
for all $f\in A^2_\alpha(\Omega)$. Thus $K_\alpha(z_j)\sqrt{\eta(z_j)}\le {\rm const.}$, contradictory.

Now we construct the desired function $f$. Pick a non-decreasing sequence of compact subsets $\{K_j\}$ of $\Omega$ such that $D=\cup K_j$. Fix a dense sequence $\{z_j\}\subset \Omega$. We reorder the points of the sequence as follows
$$
z_1,z_1,z_2,z_1,z_2,z_3,z_1,\cdots
$$
and denote the new sequence by $\{w_j\}$. Put $B_j=B(w_j,\delta_\Omega(w_j))$ where $B(z,r)$ is the euclidean ball with center $z$ and radius $r$. By the above claim, we may construct inductively sequences
$$
\{j_\nu\}\subset {\mathbb Z}^+,\ \ \ \{\zeta_\nu\}\subset \Omega,\ \ \ \{\theta_\nu\}\subset {\mathbb R},\ \ \ \{f_\nu\}\subset A^2_\alpha(\Omega)
$$
such that
$$
\zeta_\nu\in (B_\nu\backslash K_{j_\nu})\cap K_{j_{\nu+1}},\ \ \ \|f_\nu\|=1,\ \ \ \left|\sum_{\mu=1}^\nu \frac{f_\mu(\zeta_\nu)e^{i\theta_\nu}}{\mu^3(1+\|f_\mu\|_{K_{j_\mu}})}\right|\ge \frac{\nu}{\sqrt{\eta(\zeta_\nu)}}
$$
where $\|f_\mu\|_{K_{j_\mu}}=\sup_{K_{j_\mu}}|f_\mu|$. It suffices to take $f(z)=\sum_{\nu=1}^\infty \frac{f_\nu(z) e^{i\theta_\nu}}{\nu^3(1+\|f_\nu\|_{K_{j_\nu}})}$.  Q.E.D.
\end{proof}

\medskip

 Now we prove Theorem~\ref{th:levi}. The argument is essentially same as \cite{ChenFu11}. Fix first an arbitrary  point $w$ sufficiently close to $\partial \Omega$. Put $g_j=\max\{g_\Omega(\cdot,w),-j\}$, $j=1,2,\cdots$. Since $\Omega$ is hyperconvex, $g_j$ is continuous on $\Omega$ and $g_j\downarrow g_\Omega(\cdot,w)$ as $j\rightarrow \infty$. By Richberg's theorem (cf. \cite{Richberg68}), there is a $C^\infty$ strictly psh function $\psi_j<0$ on $\Omega$ such that $|\psi_j(z)-g_j(z)|<1/j$, $z\in \Omega$. Put
$$
\varphi=2ng_\Omega(\cdot,w)-\log(-g_\Omega(\cdot,w)+1),\ \ \ \varphi_j=2n\psi_j-\log(-\psi_j+1).
$$
Let $\chi:{\mathbb R}\rightarrow [0,1]$ be a $C^\infty$ cut-off function satisfying $\chi|_{(-\infty,-1)}=1$ and $\chi|_{(-\log 2,\infty)}=0$. Put
$$
v_j=\bar{\partial}\chi(-\log(-\psi_j))\frac{K_\Omega(\cdot,w)}{\sqrt{K_\Omega(w)}}
$$
where $K_\Omega$ denotes the unweighted Bergman kernel of $\Omega$.
By Theorem 1.1, there is a solution of the equation $\bar{\partial}u_j=v_j$ such that
\begin{eqnarray*}
\int_\Omega |u_j|^2e^{-\varphi_j}\delta_\Omega^{-\alpha}dV & \le & {\rm const}_{\alpha,\Omega} \int_\Omega |v_j|^2_{i\partial\bar{\partial}\varphi_j}e^{-\varphi_j}\delta_\Omega^{-\alpha}dV\\
& \le & {\rm const}_{\alpha,\Omega} \int_{{\rm supp\,}\bar{\partial}\chi(\cdot)}\frac{|K_\Omega(\cdot,w)|^2}{K_\Omega(w)}\delta_\Omega^{-\alpha} dV
\end{eqnarray*}
where the second inequality follows from
$$
i\partial\bar{\partial}\varphi_j\ge \frac{i\partial \psi_j\wedge \bar{\partial}\psi_j}{(-\psi_j+1)^2}.
$$
By Blocki's theorem, we have
$$
{\rm supp\,}\bar{\partial}\chi(\cdot)\subset \{\psi_j\le -2\}\subset \{g_\Omega(\cdot,w)\le -1\}\subset \{C^{-1}\delta_\Omega(w)|\log\delta_\Omega(w)|^{-\frac1a}\le \delta_\Omega\},\ \ \ j\gg 1,
$$
where $a$ is a Diederich-Fornaess exponent for $\Omega$. Thus
$$
\int_\Omega |u_j|^2e^{-\varphi_j}\delta_\Omega^{-\alpha}dV\le {\rm const}_{\alpha,\Omega}\, \frac{|\log \delta_\Omega(w)|^{\frac{\alpha}a}}{\delta_\Omega(w)^\alpha}.
$$
Let $u$ be a weak limit of a subsequence of $\{u_j\}$. Thus
$$
f:=\chi(-\log(-g_\Omega(\cdot,w)))K_\Omega(\cdot,w)/\sqrt{K_\Omega(w)}-u
$$
is holomorphic on $\Omega$. Since $u$ is holomorphic in a neighborhood of $w$ and
$$
\int_\Omega |u|^2 e^{-\varphi}\delta_\Omega^{-\alpha}dV\le {\rm const}_{\alpha,\Omega}\,  \frac{|\log \delta_\Omega(w)|^{\frac{\alpha}a}}{\delta_\Omega(w)^\alpha},
$$
we conclude that $u(w)=0$. Thus $f(w)=\sqrt{K_\Omega(w)}$ and
$$
\int_\Omega |f|^2 \delta_\Omega^{-\alpha}dV \le {\rm const}_{\alpha,\Omega}\,  \frac{|\log \delta_\Omega(w)|^{\frac{\alpha}a}}{\delta_\Omega(w)^\alpha}.
$$
Thus
$$
K_{\alpha}(w)\ge \frac{|f(w)|^2}{\int_\Omega |f|^2 \delta_\Omega^{-\alpha}dV }\ge {\rm const}_{\alpha,\Omega}\,K_\Omega(w)\frac{\delta_\Omega(w)^\alpha}{|\log\delta_\Omega(w)|^{\frac{\alpha}a}}\ge \frac{{\rm const}_{\alpha,\Omega}}{\delta_\Omega(w)^{2-\alpha}|\log \delta_\Omega(w)|^{\frac{\alpha}a}}
$$
as $w\rightarrow \partial \Omega$ where the last inequality follows from the Ohsawa-Takegoshi extension theorem (cf. \cite{OhsawaTakegoshi87}). Applying Proposition 4.2 with $\eta(z)=\delta_\Omega(z)^{2-\alpha}|\log \delta_\Omega(z)|^{\frac{2\alpha}{a}}$, we conclude the proof. Q.E.D.

\section{Proof of Theorem~\ref{th:gehring}}

We follows closely along Stein's book \cite{Stein72}. For each $\zeta\in \partial \Omega$, let $\nu_\zeta$ denote the unit outward normal at $\zeta$ and $T_\zeta$ the tangent plane at $\zeta$. For each $t>0$, we define an approach region ${\mathcal A}_t(\zeta)$ with vertex $\zeta$ by
$$
{\mathcal A}_t(\zeta)=\left\{z\in \Omega:|(z-\zeta)\cdot\bar{\nu}_\zeta|<(1+t)\delta_\zeta(z),\,|z-\zeta|^2<t\delta_\zeta(z)\right\}
$$
where  $\delta_\zeta(z)=\min\{\delta_\Omega(z),d(z,T_\zeta)\}$. We shall say that $|f(z)|={\rm o}(\delta_\Omega(z)^{-\beta})$ uniformly as $z\rightarrow \zeta$ admissibly for some $\beta\ge 0$ if for each $t>0$
$$
\limsup \delta_\Omega(z)^\beta|f(z)|=0
$$
as $z\rightarrow \zeta$ from the inside of ${\mathcal A}_t(\zeta)$.
For each $\zeta_0\in \partial \Omega$ and $r>0$, we put
\begin{eqnarray*}
B_1(\zeta_0,r) & = & \left\{\zeta\in \partial\Omega:|\zeta-\zeta_0|<r\right\}\\
B_2(\zeta_0,r) & = & \left\{\zeta\in \partial \Omega:|(\zeta-\zeta_0)\cdot \bar{\nu}_{\zeta_0}|<r,\,|\zeta-\zeta_0|^2<r\right\}
\end{eqnarray*}
and
$$
f_j^\ast(\zeta_0)=\sup_{r>0}\frac1{\sigma(B_j(\zeta_0,r))}\int_{B_j(\zeta_0,r)}|f(\zeta)|d\sigma(\zeta),\ \ \ j=1,2
$$
where $f\in L^p(\partial \Omega)$ and $d\sigma$ is the surface measure for $\partial \Omega$. The maximal function is defined by
$$
(Mf)(\zeta)=(f_1^\ast)_2^\ast(\zeta).
$$

\begin{theorem}\label{th:stein}
(cf. \cite{Stein72}, see also \cite{Hormander67b}).
\begin{enumerate}
\item $\|Mf\|_p\le {\rm const}_p\,\|f\|_p$, $\forall\,f\in L^p(\partial \Omega),\ 1<p\le \infty$.

 \item  Let $u$ be a psh function on $\Omega$ which is continuous on $\overline{\Omega}$ and let $f=u|_{\partial \Omega}$. Then
$$
\sup_{z\in {\mathcal A}_t(\zeta)}|u(z)|\le {\rm const}_p\, (Mf)(\zeta).
$$
\end{enumerate}
\end{theorem}

Now choose a cover of $\Omega$ by finitely many subdomains $\Omega_0,\Omega_1,\cdots,\Omega_m\subset \Omega$ with the following properties:

(a) $\partial \Omega_j$ is $C^2$.

(c) $\partial\Omega_j-(\partial \Omega_j\cap \partial \Omega)\subset \Omega$.

(b) There exists a domain $W_j\subset \partial \Omega_j\cap \partial \Omega$ such that $\{W_j\}_{j=0}^m$ forms a cover of $\partial \Omega$.

(d) There exists an outward unit normal $\nu_j$ at a point in $\partial\Omega_j\cap \partial \Omega$ such that
$$
\overline{\Omega}_j-\varepsilon v_j\subset \Omega,\ \ \ \ \ \forall\,0\le \varepsilon\ll 1.
$$
It suffices to work on a single subdomain, say $\Omega_0$.
Let $\varepsilon_0$ be a sufficiently small number. In order to apply Gehring's method (cf. \cite{Gehring57}), we define for each $t>0$, $0<\varepsilon<\varepsilon_0/2$, $\zeta\in W_0$,
\begin{eqnarray*}
U^{(t)}_\varepsilon(\zeta) & = & \left\{z\in {\mathcal A}_t(\zeta):2\varepsilon<\delta_\zeta(z)<\varepsilon_0\right\}\\
V^{(t)}_{\varepsilon}(\zeta) & = & \left\{z\in {\mathcal A}_t(\zeta)-\varepsilon\nu_0: \delta_\zeta(z)<\frac32 \varepsilon_0\right\}.
\end{eqnarray*}

\begin{lemma}\label{lem:1}
 For each $t>0$, we may choose $\varepsilon_0>0$ so that
$$
U^{(t)}_\varepsilon(\zeta)\subset V^{(s)}_{\varepsilon}(\zeta)\subset \Omega_0,\ \ \ s:=2+4t,
$$
 for all $\varepsilon< \varepsilon_0/2$ and $\zeta\in W_0$.
 \end{lemma}

\begin{proof}
  For each $z\in U^{(t)}_\varepsilon(\zeta)$, we have $\delta_\zeta(z)>2\varepsilon$. Thus
\begin{eqnarray*}
\delta_\zeta(z+\varepsilon\nu_0) & \ge & \delta_\zeta(z)-\varepsilon> \varepsilon\\
\delta_\zeta(z+\varepsilon \nu_0) & \le & \delta_\zeta(z)+\varepsilon<\frac32 \varepsilon_0
\end{eqnarray*}
for all $\varepsilon<\varepsilon_0/2$.
Since
$$
|(z-\zeta)\cdot \bar{\nu}_\zeta|<(1+t)\delta_\zeta(z),\ \ \ |z-\zeta|<(t\delta_\zeta(z))^{1/2},
$$
we get
\begin{eqnarray*}
|(z+\varepsilon\nu_0-\zeta)\cdot \bar{\nu}_\zeta| & \le & |(z-\zeta)\cdot \bar{\nu}_\zeta|+\varepsilon< (1+t)\delta_\zeta(z)+\varepsilon\le (3+2t)\delta_\zeta(z+\varepsilon\nu_0)\\
|z+\varepsilon\nu_0-\zeta|^2 & \le & 2|z-\zeta|^2+2\varepsilon^2<2t\delta_\zeta(z)+2\varepsilon\le (2+4t)\delta_\zeta(z+\varepsilon\nu_0).
\end{eqnarray*}
Thus $z+\varepsilon\nu_0\in V_\varepsilon^{(s)}(\zeta)$ where $s=2+4t$ and we get the first inclusion in the lemma.

 On the other hand, for each $z\in V_\varepsilon^{(s)}(\zeta)$, we have $|z-\zeta|^2<s\delta_\zeta(z)\le \frac32 s\varepsilon_0$, hence $V^{(s)}_\varepsilon(\zeta)\subset \Omega_0$ for all $\varepsilon< \varepsilon_0/2$, provided $\varepsilon_0$ small enough. Q.E.D.

\medskip

For each $f\in A^2_\alpha(\Omega)$, we define
$$
u^{(t)}_\varepsilon(\zeta)=\sup_{z\in U^{(t)}_\varepsilon(\zeta)}|f(z)|\ \ \ \ \ {\rm and}\ \ \ \ \  v^{(s)}_\varepsilon(\zeta)=\sup_{z\in V^{(s)}_\varepsilon(\zeta)} |f(z)|.
$$
Put $f_\varepsilon(z)=f(z-\varepsilon \nu_0)$, $z\in \Omega_0$. Clearly, $|f_\varepsilon|$ is psh in $\Omega_0$ and continuous on $\overline{\Omega}_0$. Let $M_0f_\varepsilon$ be the corresponding maximal function on $\partial\Omega_0$. Take $0<c<1$ so that
 $$
 \Omega_0-\varepsilon \nu_0=:\Omega_0^\varepsilon\subset \Omega_{c\varepsilon}:=\left\{z\in \Omega:\delta_\Omega(z)>c\varepsilon\right\}.
 $$
Let $d\sigma_0$ and $d\sigma_{c\varepsilon}$ denote the surface measures on $\partial \Omega_0$ and $\partial \Omega_{c\varepsilon}$ respectively and let $C$ denote a generic constant which is independent of $\varepsilon$ but probably depends on $\alpha,t,s$. By Theorem 5.1 and Lemma 5.2,  we have
$$
u_\varepsilon^{(t)}(\zeta)\le v_\varepsilon^{(s)}(\zeta)\le C(M_0 f_\varepsilon)(\zeta),\ \ \ \forall\,\zeta\in W_0,
$$
so that
\begin{eqnarray*}
\int_{W_0} |u_\varepsilon^{(t)}(\zeta)|^2d\sigma_0(\zeta) & \le  & C\int_{\partial \Omega_0}|M_0 f_\varepsilon|^2 d\sigma_0
 \le  C\int_{\partial \Omega_0} |f_\varepsilon|^2d\sigma_0\\
 & = & C\int_{\partial \Omega^\varepsilon_0}|f|^2 d\sigma_{0}\le C\int_{\partial \Omega_{c\varepsilon}}|f|^2d\sigma_{c\varepsilon}
\end{eqnarray*}
because of the following

\begin{lemma}\label{lem:2}
There is a constant $C>0$ independent of $\varepsilon$ and $f$ such that
$$
\int_{\partial \Omega^\varepsilon_0}|f|^2d\sigma_0\le C \int_{\partial \Omega_{c\varepsilon}}|f|^2d\sigma_{c\varepsilon}
$$
for all sufficiently small $\varepsilon>0$.
\end{lemma}

Thus for suitable small number $c_0>0$ we have
$$
\int_{0}^{c_0}\varepsilon^{-\alpha}\int_{W_0} |u_\varepsilon^{(t)}(\zeta)|^2d\sigma_0(\zeta) d\varepsilon\le C\int_{0}^{c_0}\int_{\partial \Omega_{c\varepsilon}}|f|^2\varepsilon^{-\alpha} d\sigma_{c\varepsilon} d\varepsilon\le C
\int_{\Omega}|f|^2\delta_\Omega^{-\alpha}dV<\infty,
$$
so that for $\sigma_0-$almost every $\zeta\in W_0$,
$$
\int_{0}^{c_0}\varepsilon^{-\alpha} |u_\varepsilon^{(t)}(\zeta)|^2 d\varepsilon<\infty.
$$
Hence
$$
\int_{0}^{\varepsilon'}\varepsilon^{-\alpha} |u_\varepsilon^{(t)}(\zeta)|^2 d\varepsilon={\rm o}(1)
$$
as $\varepsilon'\rightarrow 0$. Given $z\in {\mathcal A}_t(\zeta)$, we let $\varepsilon'=\delta_\zeta(z)/2$. Since $z\in U_\varepsilon^{(t)}(\zeta)$ for each $\varepsilon<\varepsilon'$, we have $u_\varepsilon^{(t)}(\zeta)\ge |f(z)|$, thus
$$
|f(z)|={\rm o}(\delta_\zeta(z)^{-\frac{1-\alpha}2})\ \ \ \ \ {\rm uniformly}
$$
as $z\rightarrow \zeta$ from the inside of ${\mathcal A}_t(\zeta)$.  Q.E.D.
\end{proof}

\medskip

 Finally we prove Lemma~\ref{lem:2}. The argument is essentially implicit in \cite{ChenFu11}. Let $P(z,w)$, $P_\varepsilon(z,w)$, $P_0(z,w)$ and $P_{0,\varepsilon}(z,w)$ denote the Poisson kernels of $\Omega$, $\Omega_{c\varepsilon}$, $\Omega_0$ and $\Omega^\varepsilon_0$ respectively. Put
$$
g(z)=\int_{\partial \Omega_{c\varepsilon}}P_\varepsilon(z,w)|f(w)|^2d\sigma_\varepsilon(w).
$$
Then $g$ is a harmonic majorant of $|f|^2$ on $\Omega_{c\varepsilon}$. Fix a point $z_0$ in $\Omega_0$. Since $P_\varepsilon(z_0,\pi_\varepsilon^{-1}(\zeta))$ converges uniformly on $\partial \Omega$ to $P(z_0,\zeta)$ where $\pi_\varepsilon$ is the normal projection from $\partial\Omega_{c\varepsilon}$ to $\partial \Omega$,
  $$
  g(z_0)\le 2C_1 \int_{\partial \Omega_{c\varepsilon}}|f(w)|^2d\sigma_\varepsilon(w)
  $$
  for all sufficiently small $\varepsilon>0$ where $C_1=\sup_{\zeta\in \partial \Omega} P(z_0,\zeta)$. On the other hand,
\begin{eqnarray*}
g(z_0) & = & \int_{\partial \Omega^\varepsilon_0} P_{0,\varepsilon}(z_0,w)g(w)d\sigma_{0}\\
& \ge & \frac{C_2}2 \int_{\partial \Omega^\varepsilon_0} g(w)d\sigma_{0}\ge \frac{C_2}2 \int_{\partial \Omega^\varepsilon_0} |f(w)|^2d\sigma_{0}
\end{eqnarray*}
for all sufficiently small $\varepsilon>0$ where $C_2=\inf_{\zeta\in \partial \Omega_0}P_0(z_0,\zeta)$. The proof is complete. Q.E.D.

\begin{remark}
 In various studies of boundary behavior of functions in Hardy spaces, the approach region defined as above is only best possible for strongly pseudoconvex domains (see e.g., \cite{Nagel-Stein-Wainger81,Krantz92}). It is probably same in the case of weighted Bergman spaces.
 \end{remark}

\section{Proof of Theorem~\ref{th:bergmanszego}}
Let $\|\cdot\|_\alpha$ and $\|\cdot\|_{\partial \Omega}$ denote the corresponding norms of the weighted Bergman space $A^2_\alpha(\Omega)$ and the Hardy space $H^2(\Omega)$ respectively. Note first that for each $f\in H^2(\Omega)$,
 and any sufficiently small $\varepsilon_0>0$,
\begin{eqnarray*}
(1-\alpha)\int_\Omega |f|^2\delta_\Omega^{-\alpha}dV & = & (1-\alpha) \int_{\Omega_{\varepsilon_0}} |f|^2\delta_\Omega^{-\alpha}dV+(1-\alpha) \int_{\Omega\backslash \Omega_{\varepsilon_0}} |f|^2\delta_\Omega^{-\alpha}dV\\
& \le & (1-\alpha) \int_{\Omega_{\varepsilon_0}} |f|^2\delta_\Omega^{-\alpha}dV+\varepsilon_0^{1-\alpha}\sup_{0<\varepsilon<\varepsilon_0}\|f\|^2_{\partial \Omega_\varepsilon}.
\end{eqnarray*}
Applying this inequality with $f(z)=S(z,w)$ for fixed $w\in \Omega$, we get
$$
\liminf_{\alpha\rightarrow 1-}\,(1-\alpha)^{-1}K_\alpha(w)\ge \liminf_{\alpha\rightarrow 1-}\,(1-\alpha)^{-1} \frac{|f(w)|^2}{\|f\|^2_{\alpha}}=\frac{S(w)^2}{\sup_{0<\varepsilon<\varepsilon_0}\|S(\cdot,w)\|^2_{\partial \Omega_\varepsilon}}
$$
locally uniformly in $w$ and uniformly in $\varepsilon_0$. Let $S_\varepsilon$ denote the Szeg\"o kernel of $\Omega_\varepsilon$. It was proved by Boas \cite{Boas87} that $S_\varepsilon(z,w)\rightarrow S(z,w)$ locally uniformly in $z,w$ and
$$
\|S_\varepsilon(\cdot,w)-S(\cdot,w)\|_{\partial \Omega_\varepsilon}\rightarrow 0
$$
locally uniformly in $w$ as $\varepsilon\rightarrow 0+$. Thus
$$
\liminf_{\alpha\rightarrow 1-}\,(1-\alpha)^{-1}K_\alpha(w)\ge \lim_{\varepsilon_0\rightarrow 0+}\frac{S(w)^2}{\sup_{0<\varepsilon<\varepsilon_0}\|S_\varepsilon(\cdot,w)\|^2_{\partial \Omega_\varepsilon}}
=\lim_{\varepsilon_0\rightarrow 0+}\frac{S(w)^2}{\sup_{0<\varepsilon<\varepsilon_0}S_\varepsilon(w)}=S(w)
$$
locally uniformly in $w$. On the other side, for any sufficiently small $\varepsilon>0$
\begin{eqnarray*}
 && \int_{\partial\Omega_\varepsilon}\left|(1-\alpha)^{-1}K_\alpha(z,w)-S_\varepsilon(z,w)\right|^2d\sigma_\varepsilon(z)\\
 & = & (1-\alpha)^{-2}\|K_\alpha(\cdot,w)\|^2_{\partial \Omega_\varepsilon}+\|S_\varepsilon(\cdot,w)\|^2_{\partial \Omega_\varepsilon}-2(1-\alpha)^{-1}{\rm Re\,}\int_{\partial\Omega_\varepsilon}K_\alpha(z,w)\overline{S_\varepsilon(z,w)}d\sigma_\varepsilon(z)\\
 & = & (1-\alpha)^{-2}\|K_\alpha(\cdot,w)\|^2_{\partial \Omega_\varepsilon}+S_\varepsilon(w)-2(1-\alpha)^{-1}K_\alpha(w).
\end{eqnarray*}
Put $f_\alpha(z):=(1-\alpha)^{-1/2}K_\alpha(z,w)/\sqrt{K_\alpha(w)}$.
Following \cite{ChenFu11}, we introduce
$$
\lambda_\alpha(\varepsilon):=\|f_\alpha\|^2_{\partial \Omega_\varepsilon}=\int_{\partial \Omega_\varepsilon}|f_\alpha|^2 d\sigma_\varepsilon.
$$
Clearly, $\lambda_\alpha$ is continuous on $(0,a]$ for some sufficiently small $a>0$ (independent of $\alpha$). For any sufficiently small $0<\varepsilon_1<\varepsilon_2<a$, $\lambda_\alpha$ assumes the minimum at some point $\varepsilon^\ast=\varepsilon^\ast(\varepsilon_1,\varepsilon_2,\alpha)$ in $[\varepsilon_1,\varepsilon_2]$. Thus
$$
1=(1-\alpha)\|f_\alpha\|^2_\alpha\ge (1-\alpha)\int_{\varepsilon_1\le \delta_\Omega \le \varepsilon_2} |f_\alpha|^2 \delta_\Omega^{-\alpha}dV\ge \left(\varepsilon_2^{1-\alpha}-\varepsilon_1^{1-\alpha}\right)\lambda_\alpha(\varepsilon^\ast),
$$
so that
$$
\|K_\alpha(\cdot,w)\|^2_{\partial \Omega_{\varepsilon^\ast}}\le (1-\alpha)\left(\varepsilon_2^{1-\alpha}-\varepsilon_1^{1-\alpha}\right)^{-1}K_\alpha(w).
$$
Thus
\begin{eqnarray*}
&& \int_{\partial\Omega_{\varepsilon^\ast}}\left|(1-\alpha)^{-1}K_\alpha(z,w)-S_{\varepsilon^\ast}(z,w)\right|^2d\sigma_{\varepsilon^\ast}(z)\\
& \le & S_{\varepsilon^\ast}(w)-(1-\alpha)^{-1}\left(2-\left(\varepsilon_2^{1-\alpha}-\varepsilon_1^{1-\alpha}\right)^{-1}\right)K_\alpha(w)\\
& = &\left(2-\left(\varepsilon_2^{1-\alpha}-\varepsilon_1^{1-\alpha}\right)^{-1}\right)\left(S(w)-(1-\alpha)^{-1}K_\alpha(w)\right)\\
&& + \left(\left(\varepsilon_2^{1-\alpha}-\varepsilon_1^{1-\alpha}\right)^{-1}-1\right) S(w)+S_{\varepsilon^\ast}(w)-S(w).
\end{eqnarray*}
It follow that
$$
\limsup_{\varepsilon_2\rightarrow 0+}\,\limsup_{\alpha\rightarrow 1-}\,\limsup_{\varepsilon_1\rightarrow 0+}\,\int_{\partial\Omega_{\varepsilon^\ast}}\left|(1-\alpha)^{-1}K_\alpha(z,w)-S_{\varepsilon^\ast}(z,w)\right|^2d\sigma_{\varepsilon^\ast}(z)= 0
$$
locally uniformly in $w$. Let $P_\varepsilon(z,\zeta)$ denote the Poisson kernel of $\Omega_\varepsilon$. For each compact set $M$ in $\Omega$ and $z,w\in M$, we have
\begin{eqnarray*}
&& \left|(1-\alpha)^{-1}K_\alpha(z,w)-S_{\varepsilon^\ast}(z,w)\right|^2\\
 & \le &  \int_{\partial\Omega_{\varepsilon^\ast}}P_{\varepsilon^\ast}(z,\zeta)\left|(1-\alpha)^{-1}K_\alpha(\zeta,w)-S_{\varepsilon^\ast}(\zeta,w)\right|^2d\sigma_{\varepsilon^\ast}(\zeta)\\
& \le & {\rm const}_M \int_{\partial\Omega_{\varepsilon^\ast}}\left|(1-\alpha)^{-1}K_\alpha(\zeta,w)-S_{\varepsilon^\ast}(\zeta,w)\right|^2d\sigma_{\varepsilon^\ast}(\zeta)
\end{eqnarray*}
provided $\varepsilon^\ast$ sufficiently small. Thus
 $(1-\alpha)^{-1}K_\alpha(z,w)\rightarrow S(z,w)$ uniformly in $z,w\in M$ as $\alpha\rightarrow 1-$. The second assertion follows immediately from this fact and Theorem~\ref{th:vanish}. Q.E.D.

\begin{problem}
 Does $(1-\alpha)^{-1}K_\alpha(z,w)$ admit an asymptotic expansion in powers of $1-\alpha$ as $\alpha\rightarrow 1-$?
\end{problem}

\section{Proof of Theorem~\ref{th:ball} }
Let $ds^2_{{\mathbb B}^n}=\partial\bar{\partial}(-\log(1-|z|^2))$ be the Bergman metric of ${\mathbb B}^n$ and $d(z,w)$ the Bergman distance between two points $z,w$. Here we omit the factor $n+1$ in the classical definition of the Bergman metric for the sake of convenience. For each $w\in {\mathbb B}^n$, $\tau>0$ and $0<r<1$, we put
$$
{B}_\tau(w)=\left\{z\in {\mathbb B}^n: d(z,w)<\tau\right\},\ \ \ \ \ {\mathbb B}_r(w)=\left\{z\in {\mathbb B}^n:|z-w|<r\right\}.
$$
Note that
$$
B_\tau(0)={\mathbb B}_r(0)\iff \tau=\frac12 \log \frac{1+r}{1-r}.
$$
Let ${\rm vol}_B$ and ${\rm vol}_E$ denote the Bergman and Euclidean volumes respectively.

\begin{proposition}\label{prop:elementary}
The following conclusions hold:
\begin{enumerate}
\item
 For each $\tau>0$, there is a constant $C_\tau>1$ such that for each $w\in {\mathbb B}^n$,
$$
{B}_\tau(w)\subset \left\{z\in {\mathbb B}^n:C_\tau^{-1}(1-|w|)<1-|z|<C_\tau(1-|w|)\right\}.
$$
$$
C_\tau^{-1}(1-|w|)^{n+1}\le {\rm vol}_E\left({B}_\tau(w)\right)\le C_\tau (1-|w|)^{n+1}.
$$

\item
 For each $r<1$,
$$
{\rm vol }_B\left({\mathbb B}_r(0)\right)\le {\rm const}_n(1-r)^{-n}.
$$

\item
For each $\tau>0$, there is a constant $t>1$ such that for each $\zeta\in {\mathbb S}^n$ and each $w\in L_\zeta$, where $L_\zeta$ is the segment determined by $0,\zeta$, we have
 $$
 {B}_\tau(w)\subset  {\mathcal A}_{t}(\zeta).
 $$
 \end{enumerate}
 \end{proposition}

\begin{proof} $(1)$ See \cite{Zhu05}, Lemma 2.20, Lemma 1.23.

$(2)$ The Bergman volume form is
$$
{\rm const}_n(1-|z|^2)^{-n-1}dV.
$$
Thus
$$
{\rm vol}_B\left({\mathbb B}_r(0)\right)={\rm const}_n\int_0^r(1-s^2)^{-n-1}s^{2n-1}ds,
$$
from which the assertion immediately follows.

$(3)$ By \cite{Zhu05}, Lemma 2.20, there is a constant $C_\tau>0$ such that
$$
|1-z\cdot\bar{w}|<C_\tau (1-|w|),\ \ \ \ \ \forall\,z\in {B}_\tau(w).
$$
Thus
$$
|1-z\cdot\bar{\zeta}|\le |1-z\cdot\bar{w}|+\left|z\cdot\overline{(w-\zeta)}\right|\le (C_\tau+1)(1-|w|)\le t(1-|z|)
$$
for suitable $t\gg 1$ by $(i)$.  Q.E.D.
\end{proof}

\begin{definition} 
 
 (see e.g., \cite{Kanai84}). 
 A subset $\Gamma=\{w_j\}_{j=1}^\infty$ of ${\mathbb B}^n$ is said to be $\tau$-{\it separated} for $\tau>0$, if $d(w_j,w_k)\ge \tau$ for all $j\neq k$, and a $\tau-$separated subset is called maximal if no more points can be added to $\Gamma$ without breaking the condition.
\end{definition}

A  basic observation is the following

\begin{lemma}\label{lem:ball1}
 Let $\Gamma=\{w_j\}_{j=1}^\infty$ be a $\tau-$separated sequence such that $0\notin \Gamma$. For any $\varepsilon>0$, 
$$
\sum_{j=1}^\infty\frac{(1-|w_j|)^n}{\left(\log \frac1{1-|w_j|}\right)^{1+\varepsilon}} <\infty.
$$
\end{lemma}

\begin{proof} The argument is standard (compare \cite{Tsuji59}, Theorem XI.\,7 and Theorem XI.\,8). For each $0<r<1$, let $n_r$ denote the number of points $w_j$ which are contained in the ball ${\mathbb B}_r(0)=B_{\frac12\log\frac{1+r}{1-r}}(0)$. Since $\{B_{\tau/2}(w_j)\}_{j=1}^\infty$ do not overlap, we have
$$
n_r {\rm vol}_B \left(B_{\tau/2}(0)\right)\le {\rm vol}_B\left( B_{\frac12\log\frac{1+r}{1-r}+\frac{\tau}2}(0)\right)={\rm vol}_B\left({\mathbb B}_{\frac{e^{\tau}(1+r)-(1-r)}{e^{\tau}(1+r)+(1-r)}}(0)\right)\le {\rm const}_{n,\tau}(1-r)^{-n}
$$
by Proposition $7.1/(2)$. Take $r_0>0$ such that $|w_j|\ge r_0$ for each $j$. Thus
\begin{eqnarray*}
&& \sum_{|w_j|<r<1}\frac{(1-|w_j|)^n}{\left(\log \frac1{1-|w_j|}\right)^{1+\varepsilon}}   =  \int_{r_0}^r \frac{(1-s)^n}{\left(\log\frac1{1-s}\right)^{1+\varepsilon}}dn_s\\
& \le & \frac{(1-r)^n}{\left(\log\frac1{1-r}\right)^{1+\varepsilon}}n_r+\int_{r_0}^r \left(\frac{(1-s)^n}{\left(\log\frac1{1-s}\right)^{1+\varepsilon}}\right)' n_s ds\\
& \le & \frac{{\rm const}_{n,\tau}}{\left(\log\frac1{1-r}\right)^{1+\varepsilon}}+{\rm const}_{n,\tau,\varepsilon}\int_{r_0}^r\frac{1}{(1-s)\left(\log\frac1{1-s}\right)^{1+\varepsilon}}ds={\rm O}(1)
\end{eqnarray*}
as $r\rightarrow 1-$. Q.E.D.
\end{proof}

\begin{lemma}\label{lem:ball2}
There is a constant $C_n>0$ such that for each $\alpha<1$, $\varepsilon>0$ and each $2\tau-$separated sequence $\Gamma=\{w_j\}_{j=1}^\infty$ with $0\notin \Gamma$ and $\tau\ge \frac{C_n}{\sqrt{1-\alpha}}$, there exists a function $f\in A^2_\alpha({\mathbb B}^n)$ such that
$$
f(w_j)=(1-|w_j|)^{-\frac{1-\alpha}2}\left(\log \frac1{1-|w_j|}\right)^{-\frac{1+\varepsilon}2},\ \ \ \ \ \forall\,j.
$$
\end{lemma}

\begin{proof} Take a $C^\infty$ cut-off function $\chi:{\mathbb R}\rightarrow [0,1]$ such that $\chi|_{(-\infty,1/4)}=1$, $\chi|_{(1/2,+\infty)}=0$ and $\chi'\le 0$. Put $d_j(z)=d(z,w_j)$ and
\begin{eqnarray*}
\psi(z) & = & \sum_j \chi(d_j(z)/\tau)\log d_j(z)/\tau\\
\varphi(z) & = & -\frac{1-\alpha}2\log(1-|z|^2)+2n\psi(z).
\end{eqnarray*}
A straightforward calculation shows
\begin{eqnarray}
\partial\bar{\partial} \psi & = & \sum_j \chi''(\cdot)\frac{\partial d_j\wedge \bar{\partial}d_j}{\tau^2} \log d_j/\tau+2\chi'(\cdot)\frac{\partial d_j\wedge \bar{\partial}d_j}{\tau d_j}\nonumber\\
&& +\chi'(\cdot)\frac{\partial\bar{\partial}d_j}\tau \log d_j/\tau+\chi(\cdot)\partial\bar{\partial}\log d_j.
\end{eqnarray}
Since $ds^2_{{\mathbb B}^n}$ has negative Riemannian sectional curvature, it follows from \cite{GreenWu79} that $\log d_j$ is psh (so is $d_j$) on ${\mathbb B}^n$. Neglecting the last two semipositive terms in (8), we get
$$
\partial\bar{\partial}\psi\ge -\frac{C_n^2}{8n\tau^2}ds^2_{{\mathbb B}^n}
$$
for suitable constant $C_n>0$. If $\tau\ge C_n/\sqrt{1-\alpha}$, then
$$
\partial\bar{\partial}\varphi\ge \frac{1-\alpha}{4} ds^2_{{\mathbb B}^n}.
$$
By Theorem 1.1, we may solve the equation
$$
\bar{\partial}u=\sum_j (1-|w_j|)^{-\frac{1-\alpha}2}\left(\log \frac1{1-|w_j|}\right)^{-\frac{1+\varepsilon}2}\bar{\partial}\chi(d_j/\tau)=:v
$$
such that
\begin{eqnarray*}
 && \int_{{\mathbb B}^n} |u|^2 e^{-\varphi} (1-|z|)^{-\frac{1+\alpha}2}dV  \le  {\rm const}_{n,\alpha} \int_{{\mathbb B}^n} |v|^2_{i\partial\bar{\partial}\varphi}e^{-\varphi}(1-|z|)^{-\frac{1+\alpha}2}dV\\
& \le & {\rm const}_{n,\alpha,\tau} \sum_j (1-|w_j|)^{-1+\alpha}\left(\log \frac1{1-|w_j|}\right)^{-1-\varepsilon}\int_{ B_\tau(w_j)}(1-|z|)^{-\alpha}dV\\
& \le & {\rm const}_{n,\alpha,\tau} \sum_{j=1}^\infty\frac{(1-|w_j|)^n}{\left(\log \frac1{1-|w_j|}\right)^{1+\varepsilon}}<\infty
\end{eqnarray*}
where the last inequality follows from Proposition $7.1/(1)$. To get the desired function, we only need to take
$$
f:=\sum_j \chi(d_j/\tau)(1-|w_j|)^{-\frac{1-\alpha}2}\left(\log \frac1{1-|w_j|}\right)^{-\frac{1+\varepsilon}2}-u.
$$
Q.E.D.
\end{proof}

Now we prove Theorem~\ref{th:ball}. Take $\tau=C_n/\sqrt{1-\alpha}$ as in Lemma~\ref{lem:ball2}. Pick a maximal $2\tau-$separated sequence $\Gamma=\{w_j\}_{j=1}^\infty$ with $0\notin \Gamma$. It is easy to see that the geodesic balls ${B}_\tau(w_j)$ are disjoint and $\{{B}_{3\tau}(w_j)\}_{j=1}^\infty$ forms a cover of ${\mathbb B}^n$. In particular,
$$
{B}_{4\tau}(w)\cap \Gamma\neq \emptyset,\ \ \ \ \ \forall\,w\in {\mathbb B}^n.
$$
By Proposition $7.1/(3)$ and completeness of $ds^2_{{\mathbb B}^n}$, we conclude that there is a constant $t>1$ such that for each $\zeta\in {\mathbb S}^n$, the set ${\mathcal A}_{t}(\zeta)$ contains a sequence of disjoint geodesic balls of radius $4\tau$ whose centers approach $\zeta$. Consequently, this set contains a subsequence of $\Gamma$. On the other hand, there is a function $f\in A^2_\alpha({\mathbb B}^n)$ such that
$$
f(w_j)=(1-|w_j|)^{-\frac{1-\alpha}2}\left(\log \frac1{1-|w_j|}\right)^{-\frac{1+\varepsilon}2},\ \ \ \forall\,j
$$
by virtue of Lemma~\ref{lem:ball2}. Thus the proof is complete. Q.E.D.

\section{Proof of Theorem~\ref{th:vanish}, \ref{th:hyperconvex}}
Let $dz=dz_1\wedge \cdots \wedge dz_n$ and $\widehat{d\bar{z}_j}=d\bar{z}_1\wedge \cdots \wedge d\bar{z}_{j-1}\wedge d\bar{z}_{j+1}\wedge \cdots \wedge d\bar{z}_n$. The Bochner-Martinelli kernel is defined to be
$$
K_{BM}(\zeta-z)=\frac{(n-1)!}{(2\pi i)^n}\sum_{j=1}^n \frac{(-1)^{j-1}(\bar{\zeta}_j-\bar{z}_j)}{|\zeta-z|^{2n}} \widehat{d\bar{\zeta}_j}\wedge d\zeta.
$$

\textbf{Bochner-Martinelli Formula.} {\it Let $D\subset {\mathbb C}^n$ be a bounded domain with $C^1-$boundary. Let $f\in C^1(\overline{D})$. Then for each $z\in D$,}
$$
f(z)=\int_{\partial D}f(\zeta) K_{BM}(\zeta-z)-\frac{(n-1)!}{(2\pi i)^n} \int_D \sum_{j=1}^n (\bar{\zeta}_j-\bar{z}_j)\frac{\partial f}{\partial \bar{\zeta}_j}\frac{d\bar{\zeta}\wedge d\zeta}{|\zeta-z|^{2n}}.
$$

First we prove Theorem~\ref{th:vanish}. Without loss of generality, we assume that the diameter $d(\Omega)$ of $\Omega$ is less than $1/2$.

(a) Put $\delta(z):=d(z,\partial \Omega),\,z\in {\mathbb C}^n$. Clearly,
$
|\delta(z)-\delta(w)|\le |z-w|
$
for all $z,w\in {\mathbb C}^n$. To apply the B-M formula, we need to approximate $\delta(z)$ first by $C^1-$smooth functions with uniformly bounded gradients by a standard argument as follows. Let $\kappa\ge 0$ be a $C^\infty$ function in ${\mathbb C}^n$ satisfying the following properties: $\kappa$ depends only on $|z|$, ${\rm supp\,}\kappa\subset {\mathbb B}^n$ and $\int_{{\mathbb C}^n}\kappa(z)dV=1$. For each $\varepsilon>0$, we put $\kappa_\varepsilon(z)=\varepsilon^{-2n}\kappa(z/\varepsilon)$ and $\delta_\varepsilon=\delta\ast \kappa_\varepsilon$. Clearly, $\delta_\varepsilon$ converges uniformly on $\overline{\Omega}$ to $\delta$, and the gradient $\bigtriangledown \delta_\varepsilon$ of $\delta_\varepsilon$ verifies
$$
\bigtriangledown \delta_\varepsilon(z)=\int_{{\mathbb C}^n}\delta(\zeta)\bigtriangledown_z \kappa_\varepsilon(\zeta-z)dV_\zeta=\int_{{\mathbb C}^n}(\delta(\zeta)-\delta(z))\bigtriangledown_z\kappa_\varepsilon(\zeta-z)dV_\zeta
$$
because $\int_{{\mathbb C}^n}\kappa_\varepsilon(\zeta-z)dV_\zeta=1$. Thus
$$
\left|\bigtriangledown\delta_\varepsilon(z)\right|\le \int_{{\mathbb C}^n}|\delta(\zeta)-\delta(z)|\cdot\left|\bigtriangledown_z \kappa_\varepsilon(\zeta-z)\right|dV_\zeta\le {\rm const}_n.
$$
Let $f\in {\mathcal O}(\Omega)$ and $z_0\in \Omega$ arbitrarily fixed. For any sufficiently small $\varepsilon>0$, there is a positive number $\varepsilon_1$ such that
$$
\left\{z\in \Omega:\varepsilon\le \delta_{\varepsilon_1}(z)\le \sqrt{\varepsilon}\right\}\subset \Omega_{\frac{\varepsilon}2}\backslash \Omega_{2\sqrt{\varepsilon}}
$$
and $\delta_{\varepsilon_1}\asymp \delta_\Omega$ holds on $\Omega_{\frac{\varepsilon}2}\backslash \Omega_{2\sqrt{\varepsilon}}$ (with implicit constants independent of $\varepsilon,\varepsilon_1$).
Now take a cut-off function $\chi$ on ${\mathbb R}$ such that $\chi|_{(-\infty,-\log 2)}=1$ and $\chi|_{(0,\infty)}=0$. Applying the B-M formula to the function
$$
\chi(\log \log 1/\delta_{\varepsilon_1}-\log\log 1/\varepsilon)f^2
$$
 with $\varepsilon$ sufficiently small, we obtain
$$
f^2(z_0)=-\frac{(n-1)!}{(2\pi i)^n} \int_\Omega \frac{f^2(\zeta)\chi'(~\cdot~)}{\delta_{\varepsilon_1}(\zeta)\log \delta_{\varepsilon_1}(\zeta)}\sum_{j=1}^n (\bar{\zeta}_j-\bar{z}_{0,j})\frac{\partial \delta_{\varepsilon_1}}{\partial \bar{\zeta}_j}(\zeta)\frac{d\bar{\zeta}\wedge d\zeta}{|\zeta-z_0|^{2n}}.
$$
Thus
$$
|f(z_0)|^2 \le  {\rm const}_{n,z_0} \int_{\Omega_{\frac{\varepsilon}2}\backslash \Omega_{2\sqrt{\varepsilon}}}|f|^2\delta_\Omega^{-1}|\log \delta_\Omega|^{-1} dV\rightarrow 0\ \ \ (\varepsilon\rightarrow 0+)
$$
provided
$$
\int_{\Omega}|f|^2\delta_\Omega^{-1}|\log \delta_\Omega|^{-1} dV<\infty.
$$

(b) Recall first that for each compact set $M \subset \Omega$, the capacity of $M$ in $\Omega$ is defined by
$$
{\rm cap}\left(M,\Omega\right)=\inf \int_\Omega \left|\bigtriangledown\phi\right|^2 dV
$$
where the infimum is taken over all $\phi\in C^\infty_0(\Omega)$ such that $0\le \phi\le 1$ and $\phi=1$ in a neighborhood of $M$. For each $j$, we may choose a function $\phi_j\in C_0^\infty(\Omega)$ with $0\le \phi_j\le 1$, $\phi_j=1$ in a neighborhood of $\overline{\Omega}_{\varepsilon_j}$, so that
$$
\int_\Omega \left|\bigtriangledown\phi_j\right|^2 dV\le 2c(\varepsilon_j).
$$
Let $f\in A^2_\alpha(\Omega)$ and $z_0\in \Omega$ arbitrarily fixed. Applying the B-M formula to the function $\phi_j f$ with $j$ sufficiently large, we get
$$
f(z_0)=-\frac{(n-1)!}{(2\pi i)^n} \int_\Omega f(\zeta)\sum_{k=1}^n (\bar{\zeta}_k-\bar{z}_{0,k})\frac{\partial \phi_j}{\partial \bar{\zeta}_k}(\zeta)\frac{d\bar{\zeta}\wedge d\zeta}{|\zeta-z_0|^{2n}}
$$
so that
\begin{eqnarray*}
|f(z_0)| & \le & {\rm const}_{n,z_0} \int_\Omega \left|\bigtriangledown \phi_j\right| |f|\,dV\\
& \le & {\rm const}_{n,z_0} \left(\int_{\Omega\backslash \Omega_{\varepsilon_j}} \left|\bigtriangledown \phi_j\right|^2\delta_\Omega^\alpha\, dV\right)^{1/2} \left(\int_{\Omega\backslash \Omega_{\varepsilon_j}} |f|^2\delta_\Omega^{-\alpha} dV\right)^{1/2}\\
& \le & {\rm const}_{n,z_0}\, c(\varepsilon_j)^{1/2}\varepsilon_j^{\alpha/2} \left(\int_{\Omega\backslash \Omega_{\varepsilon_j}} |f|^2\delta_\Omega^{-\alpha}dV\right)^{1/2}\rightarrow 0
\end{eqnarray*}
as $j\rightarrow \infty$. Q.E.D.

\medskip

On the other side, we have

\begin{proposition}\label{prop:bound}
 Let $\Omega\subset {\mathbb C}^n$ be a bounded domain and put $V(\varepsilon)={\rm vol}_E\left(\Omega\backslash \Omega_\varepsilon\right)$. If
 $$
 \alpha<\liminf_{\varepsilon\rightarrow 0+}\frac{\log V(\varepsilon)}{\log \varepsilon},
 $$
  then $H^\infty(\Omega)\subset A^2_\alpha(\Omega)$.
  \end{proposition}

\begin{proof}  It suffices to show that $1\in A^2_\alpha(\Omega)$. Fix $\beta$ such that $\alpha<\beta<\liminf_{\varepsilon\rightarrow 0+}\frac{\log V(\varepsilon)}{\log \varepsilon}$. Note that
$$
{\rm vol}_{E}(\Omega\backslash \Omega_\varepsilon)<{\rm const}_\beta\,\varepsilon^\beta
$$
for all $\varepsilon>0$. Without loss of generality, we assume $\delta_\Omega<1$ on $\Omega$ and $\alpha\ge 0$. Then we have
\begin{eqnarray*}
\int_\Omega \delta_\Omega^{-\alpha} dV & \le & \sum_{j=0}^\infty \int_{\Omega_{2^{-j-1}}\backslash \Omega_{2^{-j}}}2^{\alpha(j+1)}dV\le \sum_{j=0}^\infty 2^{\alpha(j+1)}{\rm vol}_{E}(\Omega\backslash \Omega_{2^{-j}})\\
&\le & {\rm const}_{\alpha,\beta} \sum_{j=0}^\infty 2^{-(\beta-\alpha)j}<\infty.
\end{eqnarray*}
Q.E.D.
\end{proof}

\medskip

It is reasonable to introduce the following

\begin{definition}
Let $\Omega$ be a bounded domain in ${\mathbb C}^n$. The critical exponent $\alpha(\Omega)$ of $\Omega$ for weighted Bergman spaces $A^2_\alpha(\Omega)$ is defined to be
$$
\alpha(\Omega):=\sup\left\{\alpha: A^2_\alpha(\Omega)\neq \{0\}\right\}=\inf\left\{\alpha: A^2_\alpha(\Omega)=\{0\}\right\}.
$$
\end{definition}

From Proposition 8.1 and Theorem~\ref{th:vanish}, we know that
$$
\beta(\Omega):=\liminf_{\varepsilon\rightarrow 0+}\frac{\log V(\varepsilon)}{\log \varepsilon}
\le \alpha(\Omega)\le \min\left\{1,\liminf_{\varepsilon\rightarrow 0+} \frac{\log c(\varepsilon)}{\log 1/\varepsilon}\right\}=:\gamma(\Omega).
$$
Note that $2n-\beta(\Omega)$ is nothing but the classical Minkowski dimension of $\partial \Omega$. Thus $\alpha(\Omega)=1$ in case $\partial \Omega$ is non-fractal, i.e., $\beta(\Omega)=1$. This is the case for instance, when $\Omega$ is a bounded domain in ${\mathbb C}^n$ with Lipschitz boundary or a domain in ${\mathbb C}$ whose boundary is a rectifiable Jordan curve. Unfortunately, the author is unable to find an example with $\alpha(\Omega)<1$.

\medskip

Finally we prove Theorem~\ref{th:hyperconvex}. Without loss of generality, we may assume that $\rho> -e^{-1}$ and $d(\Omega)\le 1/2$. Suppose on the contrary there is a continuous psh function $\rho<0$ on $\Omega$ such that
 $$
 -\rho\le {\rm const}_\varepsilon \delta_\Omega\left|\log \delta_\Omega\right|^{-\varepsilon}.
 $$
 Then we have
\begin{equation}
(-\rho)(-\log(-\rho))^{1+\varepsilon/2}\le {\rm const}_{\varepsilon}\delta_\Omega|\log \delta_\Omega|.
\end{equation}
By Richberg's theorem, we may also assume that $\rho$ is $C^\infty$ and strictly psh on $\Omega$. Fix $z_0\in \Omega$. Put $\phi=-\log(-\rho)$ and
   $$
\varphi(z)=2n\log |z-z_0|,\ \ \  \psi=\phi-\frac{\varepsilon}2 \log\phi.
$$
Note that $\bar{\partial}\psi=\bar{\partial}\phi-\frac{\varepsilon}2\frac{\bar{\partial}\phi}{\phi}$ and
$$
i\partial\bar{\partial} \psi=\left(1-\frac{\varepsilon}{2\phi}\right)i\partial\bar{\partial}\phi+\frac{\varepsilon}2\frac{i\partial\phi\wedge \bar{\partial}\phi}{\phi^2}\ge \left(1-\frac{\varepsilon}{2\phi}+\frac{\varepsilon}{2\phi^2}\right)i\partial\phi\wedge \bar{\partial}\phi,
$$
so that
\begin{equation}
|\bar{\partial} \psi|^2_{i\partial\bar{\partial}\psi}\le \frac{1-\frac{\varepsilon}{\phi}+\frac{\varepsilon^2}{4\phi^2}}{1-\frac{\varepsilon}{2\phi}+\frac{\varepsilon}{2\phi^2}}.
\end{equation}
Let $\chi$ be as in the proof of Theorem~\ref{th:vanish} and put $v=\bar{\partial}\chi(2|z-z_0|/\delta_\Omega(z_0)-1)$.
We need to solve the equation $\bar{\partial}u=v$ on $\Omega$ together with a Donnelly-Fefferman type estimate by using a trick from Berndtsson-Charpentier \cite{BerndtssonCharpentier00} essentially as \cite{Chen11}. Let $m>0$ be sufficiently large and $u_m$ the minimal solution of $\bar{\partial}u=v$ in $L^2(\Omega_{1/m},\varphi)$.  Then we have $u_m e^{\psi}\bot {\rm Ker\,}\bar{\partial}$ in $L^2(\Omega_{1/m},\varphi+\psi)$. Thus by H\"ormander's estimate (1.1),
\begin{eqnarray*}
\int_{\Omega_{1/m}} |u_m|^2 e^{-\varphi+\psi} dV & \le & \int_{\Omega_{1/m}} |\bar{\partial}(u_m e^{\psi})|^2_{i\partial\bar{\partial}(\varphi+\psi)} e^{-\varphi-\psi}dV\\
& \le & \int_{\Omega_{1/m}} |v+\bar{\partial}\psi\wedge u_m|^2_{i\partial\bar{\partial}\psi}e^{-\varphi+\psi}dV\\
& \le & \int_{\Omega_{1/m}} \left(1+\frac{4\phi}{\varepsilon}\right)|v|^2_{i\partial\bar{\partial}\psi}e^{-\varphi+\psi}dV+\int_{\Omega_{1/m}} \left(1+\frac{\varepsilon}{4\phi}\right)|\bar{\partial}\psi|^2_{i\partial\bar{\partial}\psi } |u_m|^2e^{-\varphi+\psi}dV.
\end{eqnarray*}
Together with (8.2), we get
\begin{equation}
\int_{\Omega_{1/m}} |u_m|^2 \phi^{-1} e^{-\varphi+\psi}dV\le {\rm const}_{\varepsilon}\int_{\Omega} \left(1+\frac{4\phi}{\varepsilon}\right)|v|^2_{i\partial\bar{\partial}\psi}e^{-\varphi+\psi}dV<\infty,
\end{equation}
for we can make $\phi$ sufficiently large if $\rho$ is replaced by $\rho/C$ with $C\gg 1$.

Now put $f_m(z):=\chi(2|z-z_0|/\delta_\Omega(z_0)-1)-u_m(z)$. Let $f$ be a weak limit of $\{f_m\}_{m=1}^\infty$. Clearly, $f\in {\mathcal O}(\Omega)$, $f(z_0)=1$ and by (8.1), (8.3),
$$
\int_\Omega |f|^2\delta_\Omega^{-1}|\log\delta_\Omega|^{-1}dV\le {\rm const}_\varepsilon \int_\Omega |f|^2 \phi^{-1} e^{\psi}dV<\infty.
$$
This contradicts with Theorem~\ref{th:vanish}. Q.E.D.

\medskip

\textbf{Acknowledgement.} The author thanks Dr. Xu Wang for pointing out several inaccuracies in a draft of this paper. He also thanks the referee and Professor Peter Pflug for valuable comments.


\begin{thebibliography}{99}

\bibitem{AndreottiVesentini65} A. Andreotti and E. Vesentini, {\it Carleman estimates for the Laplace-Beltrami equation in complex manifolds}, Publ. Math. IHES {\bf 25} (1965), 81--130.

\bibitem{Beatrous85} F. Beatrous, {\it $L^p$ estimates for extensions of holomorphic functions}, Mich. J. Math. {\bf 32} (1985), 361--380.

\bibitem{BellBoas} S. R. Bell and H. P. Boas, {\it Regularity of the Bergman projection and duality of holomorphic function spaces}, Math. Ann. {\bf 267} (1984), 473--478.

\bibitem{Berndtsson96} B. Berndtsson, {\it The extension theorem of Ohsawa-Takegoshi and the theorem of Donnelly-Fefferman}, Ann. Inst. Fourier (Grenoble) {\bf 46} (1996), 1083--1094.

\bibitem{Berndtsson01}
----------, {\it Weighted estimates for the $\bar{\partial}$-equation}, Complex Analysis and Complex Geometry (J. D. McNeal eds.),  de Gruyter, pp. 43--57, 2001.

\bibitem{BerndtssonCharpentier00} B. Berndtsson and Ph. Charpentier, \emph{A Sobolev mapping property of the Bergman knernel}, Math. Z. \textbf{235} (2000), 1-10.

\bibitem{Blocki04}
Z. Blocki, {\it The Bergman metric and the pluricomplex Green function}, Trans. Amer. Math. Soc. {\bf 357} (2004), 2613--2625.

\bibitem{Boas87} H. P. Boas, {\it The Szeg\"o projection: Sobolev estimates in regular domains}, Trans. Amer. Math. Soc. {\bf 300} (1987), 109--132.

\bibitem{BoasStraube99} H. P. Boas and E. J. Straube, {\it Global regularity of the $\bar{\partial}-$Neumann problem: a survey of the $L^2-$Sobolev theory}, In: Several Complex Variables (MSRI 1995--96), Cambrigde University Press, Cambridge 1999, 79--111.

\bibitem{Catlin80} D. Catlin, {\it Boundary behavior of holomorphic functions on pseudoconvex domains}, J. Diff. Geom. {\bf 15} (1980), 605--625.

\bibitem{Chen11} B.-Y. Chen, {\it A simple proof of the Ohsawa-Takegoshi extension theorem}, arXiv:1105.2430v1.

\bibitem{ChenFu11} B.-Y. Chen and S. Fu, {\it Comparison of the Bergman and Szeg\"o kernels}, Adv. Math. {\bf 228} (2011), 2366--2384.

\bibitem{Demailly82}
J.-P. Demailly, {\it Estimations $L^2$ pour l'op\'{e}rateur $\bar{\partial}$ d'un fibr\'{e} vectoriel holomorphe semi-positif au-dessus d'une vari\'{e}t\'{e} k\"{a}hl\'{e}rienne compl\`{e}te}, Ann. Sci. \'{E}cole Norm. Sup. {\bf 15} (1982), 457--511.

\bibitem{DonnellyFefferman83}
H. Donnelly and C. Fefferman, \emph{$L^2$-cohomology and index theorem for the Bergman metric}, Ann. of Math. (2) \textbf{118}(1983), 593--618.

\bibitem{DiederichFornaess77}
K. Diederich and J. E. Forn{\ae}ss, {\it Pseudoconvex domains: bounded strictly plurisubharmonic exhaustion functions}, Invent. Math. {\bf 39} (1977), 129--141.

\bibitem{DiederichFornaess77MathAnn}
----------, \emph{Pseudoconvex domains: an example with nontrivial Nebenh\"{u}lle}, Math. Ann. \textbf{225} (1977), 275--292.

\bibitem{DiederichOhsawa95}  K. Diederich and T. Ohsawa, {\em An estimate for the Bergman
distance on pseudoconvex domains}, Ann. of Math. {\bf 141} (1995),
181--190.

\bibitem{Englis08} M. Englis, {\it Toeplitz operator and weighted Bergman kernels}, J. Funct. Anal. {\bf 255} (2008), 1419--1457.

\bibitem{Garnett07} J. B. Garnett, Bounded Analytic Functions (2nd ed.), GTM {\bf 236}, Academic Press, San Diego, 2007.

\bibitem{Gehring57} F. W. Gehring, {\it On the radial order of subharmonic functions}, J. Math. Soc. Japan {\bf 9} (1957), 77--79.

\bibitem{GreenWu79} R. E. Greene and H. Wu, Function Theory on Manifolds
Which Possess a Pole, Lect. Notes in Math. {\bf 699}, 1979.

\bibitem{HakimSibony80} M. Hakim and N. Sibony, {\it Spetre de $A(\overline{\Omega})$ pour des domaines born\'es faiblement pseudoconvexes r\'eguliers}, J. Funct. Anal. {\bf 37} (1980), 127--135.

\bibitem{Hormander65}
L. H\"{o}rmander, {\it $L^2-$estimates and existence theorems for the $\bar{\partial}-$equation}, Acta Math. {\bf 113} (1965), 89--152.

\bibitem{Hormander67a} ----------, {\it Generators for some rings of analytic functions}, Bull. Amer. Math. Soc. {\bf 73} (1967), 943--949.

\bibitem{Hormander67b} ----------, {\it $L^p$ estimates for (pluri-)subharmonic functions}, Math. Scand. {\bf 20} (1967), 65--78.

\bibitem{Hormander90}
----------, An Introduction to Complex Analysis in Several Variables, Third Edition, Elsevier, 1990.

\bibitem{JarnickiPflug00} M. Jarnicki and P. Pflug, Extension of Holomorphic Functions, Walter de Gruyter, 2000.

\bibitem{Kanai84} M. Kanai, {\it Rough isometries, and combinatorial approximations of geometries of non-compact riemannian manifolds}, J. Math. Soc. Japan {\bf 37} (1985), 391--413.

\bibitem{Klimek91}
M. Klimek, {Pluripotential Theory}, Oxford University Press, 1991.

\bibitem{Kohn73} J. J. Kohn, {\it Global regularity for $\bar{\partial}$ on weakly pseudoconvex manifolds}, Trans. Amer. Math. Soc. {\bf 181} (1973), 273--292.

\bibitem{Krantz92}
S. G. Krantz, Function Theory of Several Complex Variables, 2nd ed., American Mathematical Society, Providence, Rhode Island, 2001.

\bibitem{LigockaHarmonic} E. Ligocka, {\it The Sobolev spaces of harmonic functions}, Studia Math. {\bf 84} (1986), no. 1, 79--87.

\bibitem{McNeal96} J. D. McNeal, {\it On large values of $L^2$ holomorphic functions}, Math. Res. Lett. {\bf 3} (1996), 247--259.

\bibitem{MichelShaw01} J. Michel and M.-C. Shaw, {\it The $\bar{\partial}-$Neumann operator on Lipschitz pseudoconvex domains with plurisubharmonic defining functions}, Duke Math. J. {\bf 108} (2001), 421--447.

\bibitem{Nagel-Stein-Wainger81} A. Nagel, E. M. Stein and S. Waigner, {\it Boundary behavior of functions holomorphic in domains of finite type}, Proc. Nat. Acad. Sci. USA {\bf 78} (1981), 6596--6599.

\bibitem{OhsawaTakegoshi87}
T. Ohsawa and K. Takegoshi, {\it On the extension of $L^2$ holomorphic functions}, Math. Z. {\bf 195} (1987), 197--204.

\bibitem{Ohsawa01} T. Ohsawa, {On the extension of $L^2$ holomorphic functions V--effects of generalization}, Nagoya Math. J. {\bf 161} (2001), 1--21.

\bibitem{Pflug75} P. Pflug, {\it Quadratintegrable holomorphe Funktionen und die Serre Vermutung}, Math. Ann. {\bf 216} (1975), 285--288.

\bibitem{Richberg68}
R. Richberg, \emph{Stetige streng pseudokonvexe funktionen}, Math. Ann. \textbf{175} (1968), 257--286.

\bibitem{Siu96} Y.-T. Siu, {\it The Fujita conjecture and the extension theorem of Ohsawa-Takegoshi}, In: Geometric Complex Analysis (Hayama 1995), World Sci. Publ. 1996, 577--592.

\bibitem{Stein72}
E.~M.~Stein, {Boundary Behavior of Holomorphic Functions of Several Complex Variables}, Princeton University Press, Princeton, New Jersey, 1972.

\bibitem{Tsuji59} M. Tsuji, Potential Theory in Modern Function Theory,
Maruzen Co., LTD. Tokyo, 1959.

\bibitem{Zhu05} K. Zhu, Spaces of Holomorphic Functions in the Unit Ball, GTM {\bf 226}, Springer, 2005.

\end{thebibliography}
\end{document}